\documentclass[12pt]{amsart}
\usepackage{amsmath}
\usepackage{amsfonts}
\usepackage{amssymb}
\usepackage{amscd}
\usepackage{appendix}

\newtheorem{thm}{Theorem}[section]
\newtheorem{prop}{Proposition}[section]
\newtheorem{lem}{Lemma}[section]

\newtheorem{de}{Definition}[section]
\newtheorem{rem}{Remark}[section]

\numberwithin{equation}{section}

\input{tcilatex}

\begin{document}
\title[CR Optimal Gap Theorem]{An Optimal Gap Theorem in a Complete Strictly
Pseudoconvex CR $(2n+1)$-manifold}
\author{Shu-Cheng Chang}
\address{Department of Mathematics and Taida Institute for Mathematical
Sciences (TIMS), National Taiwan University, Taipei 10617, Taiwan}
\email{scchang@math.ntu.edu.tw}
\author{Yen-Wen Fan}
\address{Taida Institute for Mathematical Sciences, National Taiwan
University, Taipei 10617, Taiwan}
\email{fanyenwen@gmail.com }
\thanks{Research supported in part by MOST of Taiwan}
\subjclass{Primary 32V05, 32V20; Secondary 53C56.}
\keywords{Gap theorem, Li-Yau-Hamilton inequality, CR Hodge-Laplace, CR
Monment type estimate, Heat kernel, Subelliptic operator}

\begin{abstract}
In this paper, by applying a linear trace Li-Yau-Hamilton inequality for a
positive $(1,1)$-form solution of the CR Hodge-Laplace heat equation and
monotonicity of the heat equation deformation, we obtain an optimal gap
theorem for a complete strictly pseudocovex CR $(2n+1)$-manifold with
nonnegative pseudohermitian bisectional curvature and vanishing torsion. We
prove that if the average of the Tanaka-Webster scalar curvature over a ball
of radius $r$ centered at some point $o$ decays as $o\left( r^{-2}\right) $,
then the manifold is flat.
\end{abstract}

\maketitle

\section{Introduction}

In \cite{gw1}, \cite{s} and \cite{y}, it is conjectured that a complete
noncompact K\"{a}hler manifold of positive holomorphic bisectional curvature
of complex dimension $m$ is biholomorphic to $\mathbf{C}^{m}$. The first
result concerning this conjecture was obtained by Mok-Siu-Yau (\cite{msy})
and Mok (\cite{mok2}). Let $M$ be a complete noncompact K\"{a}hler manifold
of nonnegative holomorphic bisectional curvature of complex dimension $m\geq
2$. They proved that $M$ is isometrically biholomorphic to $\mathbf{C}^{m}$
with the standard flat metric under the assumptions of the maximum volume
growth condition 
\begin{equation*}
V_{o}\left( r\right) \geq \delta r^{2m}
\end{equation*}%
for some point $o\in M,\ \delta >0,$ $r(x)=d(o,x)$ and the scalar curvature $%
R$ decays as 
\begin{equation*}
R(x)\leq \frac{C}{1+r^{2+\varepsilon }},\ \ \ x\in M
\end{equation*}%
for $C>0$ and any arbitrarily small positive constant $\varepsilon $. Since
then there are several further works aiming to prove the optimal result and
reader is referred to \cite{mok1}, \cite{ctz}, \cite{cz}, \cite{n4} and \cite%
{nt2}. A key common ingredient used in the previous works such as \cite{msy}%
, \cite{n4} and \cite{nt2} is to solve the so-called Poincare Lelong
equation \ $\sqrt{-1}\partial \overline{\partial }u=\rho $, for a given $d$%
-closed real $(1,1)$- form $\rho $ and then show that $trace(\rho )=0$ by
using (\ref{2001a}). In particular in \cite{nt2}, Ni and Tam showed that the
solution $u(x)$ of $\sqrt{-1}\partial \overline{\partial }u=Ric$ is of $%
o(logr(x))$ growth with the extra condition $\lim \inf_{r\rightarrow \infty
}\exp \left( -ar^{2}\right) \int_{B_{o}\left( r\right) }R^{2}\left( y\right)
d\mu \left( y\right) <\infty $ for some $a>0$. Then the result follows from
the Liouville theorem for plurisubharmonic functions which asserts that any
continuous plurisubharmonic function with upper growth bound of $o(logr(x))$
must be a constant.

In 2012, L. Ni finally obtained an optimal gap theorem (\cite{n2}) on $M$
with nonnegative bisectional curvature without the maximum volume growth
condition, provided the following scalar decays 
\begin{equation}
\frac{1}{V_{o}\left( r\right) }\int_{B_{o}\left( r\right) }R\left( y\right)
d\mu \left( y\right) =o\left( r^{-2}\right) .  \label{2001a}
\end{equation}

In the paper of \cite{n2}, L. Ni adapted a different method which has also
succeeded in the recent resolution of the fundamental gap conjecture in \cite%
{ac}. The key step is , using a sharp differential estimate and monotonicity
of heat equation deformation of positive $(1,1)$-forms as in (\cite{n1}), it
provided an alternate argument of proving the above mentioned Liouville
theorem.

A Riemannian version of (\cite{msy}) was proved in \cite{gw2} shortly
afterwards. This present paper is concerned with an analogue of CR gap
theorem on a complete noncompact strictly pseudoconvex CR $(2n+1)$-manifold
with nonnegative bisectional curvature. Recently, enlightened by the work of 
\cite{n1} as above, \ we obtained the linear trace version of
Li-Yau-Hamilton inequality for positive solutions of the CR
Lichnerowicz-Laplacian heat equation and then CR monotonicity of heat
equation deformation of positive $(1,1)$-forms is available in order to
prove the following CR gap Theorem :

\begin{thm}
Let $M$ be a complete noncompact strictly pseudoconvex CR $(2n+1)$-manifold
with nonnegative bisectional curvature and vanishing torsion. Then $M$ is
flat if 
\begin{equation}
\frac{1}{V_{o}\left( r\right) }\int_{B_{o}\left( r\right) }S\left( y\right)
d\mu \left( y\right) =o\left( r^{-2}\right) ,  \label{2001}
\end{equation}%
for some point $o\in M.$ Here $S\left( y\right) $ is the Tanaka-Webster
scalar curvature and $V_{o}\left( r\right) $ is the volume of the ball $%
B_{o}\left( r\right) $ with respect to the Carnot-Carath\'{e}odory distance.
As a consequence if $M$ is not flat, then%
\begin{equation*}
\liminf_{r\longrightarrow \infty }\frac{r^{2}}{V_{o}\left( r\right) }%
\int_{B_{o}\left( r\right) }S\left( y\right) d\mu \left( y\right) >0
\end{equation*}%
for any $o\in M.$
\end{thm}

Here we adapt the method as in \cite{n2}. Below is the main idea in our
proof. We first work on degenerated parabolic systems in CR manifolds which
is different to K\"{a}hler manifolds :%
\begin{equation*}
\left \{ 
\begin{array}{lll}
\frac{\partial }{\partial t}\phi (x,t) & = & \Delta _{H}\phi (x,t), \\ 
\phi \left( x,0\right) & = & Ric(x)\geq 0.%
\end{array}%
\right.
\end{equation*}%
Here $\Delta _{H}$ is the CR Hodge-Laplacian operator, $Ric(x)=iR_{\alpha 
\bar{\beta}}\theta ^{\alpha }\wedge \theta ^{\overline{\beta }}$ is the
pseudohermitian Ricci form of a strictly pseudoconvex CR $(2n+1)$-manifold.

Let $M$ be a complete noncompact strictly pseudoconvex CR $(2n+1)$-manifold
with nonnegative bisectional curvature and vanishing torsion. It follows
from Proposition \ref{prop2-1} that there exists a long time solution $\phi
(x,t)$ with $\phi (x,t)\geq 0$ on $M\times \lbrack 0,\infty ).$ Now let $%
u(x,t)=\Lambda (\phi )$ which is nonnegative and satisfies the CR heat
equation with $u(x,0)=S(x).$ Li-Yau-Hamilton Harnack quantity (\ref{3001})
and monotonicity property (\ref{4006}) with vanishing mixed-term implies
that $tu\left( x,t\right) $ is nondecreasing in $t$ for any $x$. Finally,
the assumption (\ref{2001}) and CR moment type estimate (\ref{2013a}) imply $%
\lim_{t\rightarrow \infty }tu\left( x_{0},t\right) =0$. Hence the
monotonicity and maximum principle imply $tu\left( x,t\right) \equiv 0$ for
all $t>0$ and any $x\in M$ . The flatness then follows from $u\left(
x,0\right) =0$ which is clear by continuity.

The rest of the paper is organized as follows. In section $2$, we give an
introduction to pseudohermitian manifolds and some notations. In section $3$%
, we obtain the CR moment type estimate which is the first key estimate for
the proof of main theorem. In section $4,$ we relate the linear trace
Li-Yau-Hamilton type inequality of the CR Lichnerowicz-Laplacian heat
equation to a monotonicity formula of the heat solution. In section $5,$ we
prove the CR optimal gap Theorem.

\section{Preliminary}

First we introduce some basic materials in a pseudohermitian $(2n+1)$%
-manifold ( see \cite{l1}, \cite{l2} for more details ). Let $(M,\xi )$ be a 
$(2n+1)$-dimensional, orientable, contact manifold with contact structure $%
\xi $. A CR structure compatible with $\xi $ is an endomorphism $J:\xi
\rightarrow \xi $ such that $J^{2}=-1$. We also assume that $J$ satisfies
the following integrability condition: If $X$ and $Y$ are in $\xi $, then so
are $[JX,Y]+[X,JY]$ and $J([JX,Y]+[X,JY])=[JX,JY]-[X,Y]$.

Let $\left \{ T,Z_{\alpha },Z_{\bar{\alpha}}\right \} $ be a frame of $%
TM\otimes \mathbb{C}$, where $Z_{\alpha }$ is any local frame of $T_{1,0},\
Z_{\bar{\alpha}}=\overline{Z_{\alpha }}\in T_{0,1}$ and $T$ is the
characteristic vector field. Then $\left \{ \theta ,\theta ^{\alpha },\theta
^{\bar{\alpha}}\right \} $, which is the coframe dual to $\left \{
T,Z_{\alpha },Z_{\bar{\alpha}}\right \} $, satisfies 
\begin{equation}
d\theta =ih_{\alpha \overline{\beta }}\theta ^{\alpha }\wedge \theta ^{%
\overline{\beta }}  \label{72}
\end{equation}%
for some positive definite hermitian matrix of functions $(h_{\alpha \bar{%
\beta}})$, if we have this contact structure, we call such $M$ a strictly
pseudoconvex CR $(2n+1)$-manifold.

The Levi form $\left \langle \ ,\ \right \rangle _{L_{\theta }}$ is the
Hermitian form on $T_{1,0}$ defined by%
\begin{equation*}
\left \langle Z,W\right \rangle _{L_{\theta }}=-i\left \langle d\theta
,Z\wedge \overline{W}\right \rangle .
\end{equation*}%
We can extend $\left \langle \ ,\ \right \rangle _{L_{\theta }}$ to $T_{0,1}$
by defining $\left \langle \overline{Z},\overline{W}\right \rangle
_{L_{\theta }}=\overline{\left \langle Z,W\right \rangle }_{L_{\theta }}$
for all $Z,W\in T_{1,0}$. The Levi form induces naturally a Hermitian form
on the dual bundle of $T_{1,0}$, denoted by $\left \langle \ ,\
\right
\rangle _{L_{\theta }^{\ast }}$, and hence on all the induced tensor
bundles. Integrating the Hermitian form (when acting on sections) over $M$
with respect to the volume form $d\mu =\theta \wedge (d\theta )^{n}$, we get
an inner product on the space of sections of each tensor bundle.

The pseudohermitian connection of $(J,\theta )$ is the connection $\nabla $
on $TM\otimes \mathbb{C}$ (and extended to tensors) given in terms of a
local frame $Z_{\alpha }\in T_{1,0}$ by

\begin{equation*}
\nabla Z_{\alpha }=\theta _{\alpha }{}^{\beta }\otimes Z_{\beta },\quad
\nabla Z_{\bar{\alpha}}=\theta _{\bar{\alpha}}{}^{\bar{\beta}}\otimes Z_{%
\bar{\beta}},\quad \nabla T=0,
\end{equation*}%
where $\theta _{\alpha }{}^{\beta }$ are the $1$-forms uniquely determined
by the following equations:

\begin{equation*}
\begin{split}
d\theta ^{\beta }& =\theta ^{\alpha }\wedge \theta _{\alpha }{}^{\beta
}+\theta \wedge \tau ^{\beta }, \\
0& =\tau _{\alpha }\wedge \theta ^{\alpha }, \\
0& =\theta _{\alpha }{}^{\beta }+\theta _{\bar{\beta}}{}^{\bar{\alpha}},
\end{split}%
\end{equation*}%
We can write (by Cartan lemma) $\tau _{\alpha }=A_{\alpha \gamma }\theta
^{\gamma }$ with $A_{\alpha \gamma }=A_{\gamma \alpha }$. The curvature of
Webster-Stanton connection, expressed in terms of the coframe $\{ \theta
=\theta ^{0},\theta ^{\alpha },\theta ^{\bar{\alpha}}\}$, is 
\begin{equation*}
\begin{split}
\Pi _{\beta }{}^{\alpha }& =\overline{\Pi _{\bar{\beta}}{}^{\bar{\alpha}}}%
=d\omega _{\beta }{}^{\alpha }-\omega _{\beta }{}^{\gamma }\wedge \omega
_{\gamma }{}^{\alpha }, \\
\Pi _{0}{}^{\alpha }& =\Pi _{\alpha }{}^{0}=\Pi _{0}{}^{\bar{\beta}}=\Pi _{%
\bar{\beta}}{}^{0}=\Pi _{0}{}^{0}=0.
\end{split}%
\end{equation*}%
Webster showed that $\Pi _{\beta }{}^{\alpha }$ can be written 
\begin{equation*}
\Pi _{\beta }{}^{\alpha }=R_{\beta }{}^{\alpha }{}_{\rho \bar{\sigma}}\theta
^{\rho }\wedge \theta ^{\bar{\sigma}}+W_{\beta }{}^{\alpha }{}_{\rho }\theta
^{\rho }\wedge \theta -W^{\alpha }{}_{\beta \bar{\rho}}\theta ^{\bar{\rho}%
}\wedge \theta +i\theta _{\beta }\wedge \tau ^{\alpha }-i\tau _{\beta
}\wedge \theta ^{\alpha }
\end{equation*}%
where the coefficients satisfy 
\begin{equation*}
R_{\beta \bar{\alpha}\rho \bar{\sigma}}=\overline{R_{\alpha \bar{\beta}%
\sigma \bar{\rho}}}=R_{\bar{\alpha}\beta \bar{\sigma}\rho }=R_{\rho \bar{%
\alpha}\beta \bar{\sigma}},\ \ \ W_{\beta \bar{\alpha}\gamma }=W_{\gamma 
\bar{\alpha}\beta }.
\end{equation*}%
Here $R_{\gamma }{}^{\delta }{}_{\alpha \bar{\beta}}$ is the pseudohermitian
curvature tensor, $R_{\alpha \bar{\beta}}=R_{\gamma }{}^{\gamma }{}_{\alpha 
\bar{\beta}}$ is the pseudohermitian Ricci curvature tensor, $S=R_{\alpha 
\overline{\alpha }}$ is the Tanaka-Webster scalar curvature and $A_{\alpha
\beta }$ \ is the torsion tensor. Furthermore, we define the bi-sectional
curvature 
\begin{equation*}
R_{\alpha \bar{\alpha}\beta \overline{\beta }}(X,Y)=R_{\alpha \bar{\alpha}%
\beta \overline{\beta }}X_{\alpha }X_{\overline{\alpha }}Y_{\beta }Y_{\bar{%
\beta}}
\end{equation*}%
and the bi-torsion tensor 
\begin{equation*}
T_{\alpha \overline{\beta }}(X,Y):=i(A_{\bar{\beta}\bar{\rho}}X_{\rho
}Y_{\alpha }-A_{\alpha \rho }X_{\bar{\rho}}Y_{\bar{\beta}})
\end{equation*}%
and the torsion tensor \ 
\begin{equation*}
Tor(X,Y):=h^{\alpha \bar{\beta}}T_{\alpha \overline{\beta }}(X,Y)=i(A_{%
\overline{\alpha }\bar{\rho}}X_{\rho }Y_{\alpha }-A_{\alpha \rho }X_{\bar{%
\rho}}Y_{\overline{\alpha }})
\end{equation*}%
for any $X=X_{\overline{\alpha }}Z_{\alpha },\ Y=Y_{\overline{\alpha }%
}Z_{\alpha }$ in $T_{1,0}.$

We will denote components of covariant derivatives with indices preceded by
comma; thus write $A_{\alpha \beta ,\gamma }$. The indices $\{0,\alpha ,\bar{%
\alpha}\}$ indicate derivatives with respect to $\{T,Z_{\alpha },Z_{\bar{%
\alpha}}\}$. For derivatives of a scalar function, we will often omit the
comma, for instance, $u_{\alpha }=Z_{\alpha }u,\ u_{\alpha \bar{\beta}}=Z_{%
\bar{\beta}}Z_{\alpha }u-\omega _{\alpha }{}^{\gamma }(Z_{\bar{\beta}%
})Z_{\gamma }u.$

For a smooth real-valued function $u$, the subgradient $\nabla _{b}$ is
defined by $\nabla _{b}u\in \xi $ and $\left \langle Z,\nabla
_{b}u\right
\rangle _{L_{\theta }}=du(Z)$ for all vector fields $Z$ tangent
to contact plane. Locally $\nabla _{b}u=\sum_{\alpha }u_{\bar{\alpha}%
}Z_{\alpha }+u_{\alpha }Z_{\bar{\alpha}}$. We also denote $u_{0}=Tu$.

We can use the connection to define the subhessian as the complex linear map 
\begin{equation*}
(\nabla ^{H})^{2}u:T_{1,0}\oplus T_{0,1}\rightarrow T_{1,0}\oplus T_{0,1}
\end{equation*}%
by 
\begin{equation*}
(\nabla ^{H})^{2}u(Z)=\nabla _{Z}\nabla _{b}u.\ 
\end{equation*}%
In particular,

\begin{equation*}
|\nabla _{b}u|^{2}=2u_{\alpha }u_{\overline{\alpha }},\quad |\nabla
_{b}^{2}u|^{2}=2(u_{\alpha \beta }u_{\overline{\alpha }\overline{\beta }%
}+u_{\alpha \overline{\beta }}u_{\overline{\alpha }\beta }).
\end{equation*}

Also 
\begin{equation*}
\begin{array}{c}
\Delta _{b}u=Tr\left( (\nabla ^{H})^{2}u\right) =\sum_{\alpha }(u_{\alpha 
\bar{\alpha}}+u_{\bar{\alpha}\alpha }).%
\end{array}%
\end{equation*}%
The Kohn-Rossi Laplacian $\square _{b}$ on functions is defined by 
\begin{equation*}
\square _{b}\varphi =2\overline{\partial }_{b}^{\ast }\overline{\partial }%
_{b}\varphi =(\Delta _{b}+inT)\varphi =-2\varphi _{\overline{\alpha }}{}^{%
\overline{\alpha }}
\end{equation*}%
and on $(p,q)$-forms is defined by%
\begin{equation*}
\square _{b}=2(\overline{\partial }_{b}^{\ast }\overline{\partial }_{b}+%
\overline{\partial }_{b}\overline{\partial }_{b}^{\ast }).
\end{equation*}

Next we recall the following commutation relations (\cite{l1}). \ Let $%
\varphi $ be a scalar function and $\sigma =\sigma _{\alpha }\theta ^{\alpha
}$ be a $\left( 1,0\right) $ form, then we have

\begin{equation*}
\begin{array}{ccl}
\varphi _{\alpha \beta } & = & \varphi _{\beta \alpha }, \\ 
\varphi _{\alpha \bar{\beta}}-\varphi _{\bar{\beta}\alpha } & = & ih_{\alpha 
\overline{\beta }}\varphi _{0}, \\ 
\varphi _{0\alpha }-\varphi _{\alpha 0} & = & A_{\alpha \beta }\varphi _{%
\bar{\beta}}, \\ 
\sigma _{\alpha ,0\beta }-\sigma _{\alpha ,\beta 0} & = & \sigma _{\alpha ,%
\bar{\gamma}}A_{\gamma \beta }-\sigma _{\gamma }A_{\alpha \beta ,\bar{\gamma}%
}, \\ 
\sigma _{\alpha ,0\bar{\beta}}-\sigma _{\alpha ,\bar{\beta}0} & = & \sigma
_{\alpha ,\gamma }A_{\bar{\gamma}\bar{\beta}}+\sigma _{\gamma }A_{\bar{\gamma%
}\bar{\beta},\alpha },%
\end{array}%
\end{equation*}%
and 
\begin{equation*}
\begin{array}{ccl}
\sigma _{\alpha ,\beta \gamma }-\sigma _{\alpha ,\gamma \beta } & = & 
iA_{\alpha \gamma }\sigma _{\beta }-iA_{\alpha \beta }\sigma _{\gamma }, \\ 
\sigma _{\alpha ,\bar{\beta}\bar{\gamma}}-\sigma _{\alpha ,\bar{\gamma}\bar{%
\beta}} & = & ih_{\alpha \overline{\beta }}A_{\bar{\gamma}\bar{\rho}}\sigma
_{\rho }-ih_{\alpha \overline{\gamma }}A_{\bar{\beta}\bar{\rho}}\sigma
_{\rho }, \\ 
\sigma _{\alpha ,\beta \bar{\gamma}}-\sigma _{\alpha ,\bar{\gamma}\beta } & =
& ih_{\beta \overline{\gamma }}\sigma _{\alpha ,0}+R_{\alpha \bar{\rho}%
}{}_{\beta \bar{\gamma}}\sigma _{\rho }.%
\end{array}%
\end{equation*}

Moreover for multi-index $I=\left( \alpha _{1},...,\alpha _{p}\right) ,\ 
\bar{J}=\left( \bar{\beta}_{1},...,\bar{\beta}_{q}\right) ,$ we denote $%
I(\alpha _{k}=\mu )=\left( \alpha _{1},...,\alpha _{k-1},\mu ,\alpha
_{k+1},...,\alpha _{p}\right) .$ Then%
\begin{equation*}
\begin{array}{ccl}
\eta _{I\bar{J},\mu \lambda }-\eta _{I\bar{J},\lambda \mu } & = & 
i\sum_{k=1}^{p}\left( \eta _{I(\alpha _{k}=\mu )\bar{J}}A_{\alpha
_{k}\lambda }-\eta _{I(\alpha _{k}=\lambda )\bar{J}}A_{\alpha _{k}\mu
}\right) \\ 
&  & -i\sum_{k=1}^{q}\left( \eta _{I\bar{J}\left( \bar{\beta}_{k}=\bar{\gamma%
}\right) }h_{\bar{\beta}_{k}\mu }A_{\lambda }^{\bar{\gamma}}-\eta _{I\bar{J}%
\left( \bar{\beta}_{k}=\bar{\gamma}\right) }h_{\bar{\beta}_{k}\lambda
}A_{\mu }^{\bar{\gamma}}\right) ,%
\end{array}%
\end{equation*}%
and

\begin{equation*}
\begin{array}{ccl}
\eta _{I\bar{J},\lambda \bar{\mu}}-\eta _{I\bar{J},\bar{\mu}\lambda } & = & 
ih_{\lambda \bar{\mu}}\eta _{I\bar{J},0}+\sum_{k=1}^{p}\eta _{I\left( \alpha
_{k}=\gamma \right) \bar{J}}R_{\alpha _{k}\mathit{\ \ }\lambda \bar{\mu}}^{%
\mathit{\ \ \ }\gamma }+\sum_{k=1}^{q}\eta _{I\bar{J}\left( \bar{\beta}_{k}=%
\bar{\gamma}\right) }R_{\bar{\beta}_{k}\mathit{\ \ \ }\lambda \bar{\mu}}^{%
\mathit{\ \ \ \ }\bar{\gamma}} \\ 
\eta _{I\bar{J},0\mu }-\eta _{I\bar{J},\mu 0} & = & A_{\mu }^{\bar{\rho}%
}\eta _{I\bar{J},\bar{\rho}}-\sum_{k=1}^{p}A_{\alpha _{k}\mu ,\bar{\rho}%
}\eta _{I\left( \alpha _{k}=\rho \right) \bar{J}}+\sum_{k=1}^{q}A_{\mu \rho ,%
\bar{\beta}_{k}}\eta _{I\bar{J}\left( \bar{\beta}_{k}=\bar{\rho}\right) }.%
\end{array}%
\end{equation*}

Finally, we recall the following definition.

\begin{de}
\label{d1} A piecewise smooth curve $\gamma :[0,1]\rightarrow M$ is said to
be horizontal if $\gamma \ ^{\prime }(t)\in \xi $ whenever $\gamma \
^{\prime }(t)$ exists. The length of $\gamma $ is then defined by 
\begin{equation*}
l(\gamma )=\int_{0}^{1}\left \langle \gamma \ ^{\prime }(t),\gamma \
^{\prime }(t)\right \rangle _{L_{\theta }}^{\frac{1}{2}}dt.
\end{equation*}%
The Carnot-Carath\'{e}odory distance between two points $p,\ q\in M$ is 
\begin{equation*}
d_{c}(p,q)=\text{inf}\left \{ l(\gamma )|\ \gamma \in C_{p,q}\right \} ,
\end{equation*}%
where $C_{p,q}$ is the set of all horizontal curves joining $p$ and $q$.
\end{de}

\section{CR Moment-Type Estimates}

Let $(M,J,\theta )$ be a strictly pseudoconvex CR $(2n+1)$-manifold. In our
recent paper (\cite{cct} and \cite{ccf}), we consider the CR Hodge-Laplacian%
\begin{equation*}
\Delta _{H}=-\frac{1}{2}(\square _{b}+\overline{\square }_{b})
\end{equation*}%
for Kohn-Rossi Laplacian $\square _{b}$. For any $(1,1)$-form $\phi
(x,t)=\phi _{\alpha \overline{\beta }}\theta ^{\alpha }\wedge \theta ^{%
\overline{\beta }},$ we study the CR Hodge-Laplacian heat equation on $%
M\times \lbrack 0,T)$ 
\begin{equation}
\frac{\partial }{\partial t}\phi (x,t)=\Delta _{H}\phi (x,t).  \label{41}
\end{equation}%
It follows from the CR Bochner-Weitzenbock Formula (\cite{ccf}) that the CR
parabolic equation (\ref{41}) is equivalent to the CR analogue of
Lichnerowicz-Laplacian heat equation : 
\begin{equation}
\frac{\partial }{\partial t}\phi _{\alpha \bar{\beta}}=\Delta _{b}\phi
_{\alpha \bar{\beta}}+2R_{\alpha \bar{\gamma}\mu \bar{\beta}}\phi _{\gamma 
\bar{\mu}}-(R_{\gamma \bar{\beta}}\phi _{\alpha \bar{\gamma}}+R_{\alpha \bar{%
\gamma}}\phi _{\gamma \bar{\beta}}).  \label{70}
\end{equation}%
In this section, we consider the following Dirichlet problem of degenerate
parabolic systems :

\begin{equation}
\left \{ 
\begin{array}{llll}
\left( \frac{\partial }{\partial t}-\Delta _{H}\right) \phi & = & 0, & \text{%
on }\Omega \times \lbrack 0,\infty ), \\ 
\phi \left( x,t\right) & = & 0, & \text{on }\partial \Omega \times \lbrack
0,\infty ), \\ 
\phi \left( x,0\right) & = & \phi _{ini}\left( x\right) & \text{on }\Omega .%
\end{array}%
\right.  \label{2002}
\end{equation}

In contrast to K\"{a}hler case, the regularity of a solution for $\Delta
_{H} $ up to $\partial \Omega $ may depend on geometry around the
characteristic point at the boundary ( \cite{J1} and \cite{J2}) in the CR
setting. In fact,

\begin{prop}
\label{prop2-1}There exists "sweetsop" exhaustion domains $\Omega _{\mu }$
such that the solutions $\phi _{\mu }$ of (\ref{2002}) are $C\left(
C^{2,\alpha }\left( \bar{\Omega}_{\mu },\Lambda ^{1,1}\right) ,[0,T)\right)$.
\end{prop}

We will give a detail proof of Proposition \ref{prop2-1} in Appendix \textbf{%
A. }After the construction of the "sweetsop" exhaustion domain $\Omega _{\mu
}$ for $\Delta _{H}$ as in Proposition \ref{prop2-1}, one is able to apply
semigroup method (\cite{p}) to obtain better regularity of the solution of
the CR Lichnerowitz-Laplacian heat equation (\ref{2002}) which depends on
regularity of the initial condition. One more tensor maximum principle below
is needed in the proof of main theorem in order to have nonnegativity of the
constructed solution $\phi _{\mu }$ if the initial data is nonnegative.

\begin{prop}
\label{prop2-2} Let $(M,J,\theta )$ be a strictly pseudoconvex CR $(2n+1)$%
-manifold with nonnegative bisectional curvature. Let $\Omega $ be bounded
domain in $M.$ Assume that $\phi \left( x,t\right) $ is a $\left( 1,1\right) 
$-form satisfies%
\begin{equation*}
\left \{ 
\begin{array}{llll}
\left( \frac{\partial }{\partial t}-\Delta _{H}\right) \phi & = & 0, & \text{%
on }\Omega \times \lbrack 0,\infty ), \\ 
\phi \left( x,t\right) & \geq & 0, & \text{on }\partial \Omega \times
\lbrack 0,\infty ), \\ 
\phi \left( x,0\right) & \geq & 0 & \text{on }\Omega .%
\end{array}%
\right.
\end{equation*}%
Then $\phi \left( x,t\right) \geq 0$ on $\Omega \times \lbrack 0,\infty ).$
\end{prop}

\begin{proof}
Similar to proposition 11.1 in \cite{nn}.
\end{proof}

The first key estimate for the proof of main theorem is the moment type
estimate. This estimate is first introduced by L. Ni (\cite{n3}). By using
Li-Yau type heat kernel estimate, he proved that a nonnegative solution $%
u(x,t)$ of the heat equation are $t^{d/2}$ growth if and only if the average
function $k\left( x,r\right) :=\frac{1}{V\left( r\right) }\int_{B_{x}\left(
r\right) }f\left( y\right) dy$ of the initial data $f\left( y\right) $ grows
as $r^{d}$ in a certain complete Kaehler manifold. In our CR setting, we
only has the CR moment type estimate for a nonnegative heat solution which
can be express as $P_{t}f$ for a smooth bounded function $f$ on $M$. In
contrast to K\"{a}her case, in general, we do not know if any nonnegative
heat solution could hold.

To introduce our version, we will follow from semigroup method as in \cite{m}
( also \cite{bbgm}). It is known that the heat semigroup $\left(
P_{t}\right) _{t\geq 0}$ is given by 
\begin{equation*}
P_{t}=\int_{0}^{\infty }e^{-\lambda t}dE_{\lambda }
\end{equation*}%
for the spectral decomposition of $\Delta _{b}=-\int_{0}^{\infty }\lambda
dE_{\lambda }$ in $L^{2}\left( M\right) $. It is a one-parameter family of
bounded operators on $L^{2}\left( M\right) .$ We denote%
\begin{equation*}
P_{t}f\left( x\right) =\int_{M}p\left( x,y,t\right) f\left( y\right) d\mu
\left( y\right) ,
\end{equation*}%
\ \ \ for $f\in C_{0}^{\infty }\left( M\right) .$ Here $p\left( x,y,t\right)
>0$ is the so-called symmetric heat kernel associated to $P_{t}$. Due to
hypoellipticity of $\Delta _{b},$ the function $\left( x,t\right)
\rightarrow P_{t}f\left( x\right) $ is smooth on $M\times \left( 0,\infty
\right) .$

In the following we use $V\left( r\right) $ and $B_{x}\left( r\right) $
denote the volume of a unit ball with respect to the Carnot-Carath\'{e}odory
distance and measure $d\mu =\theta \wedge \left( d\theta \right) ^{n}$. We
recall some facts from \cite{m} ( also \cite{bg} and \cite{bbgm}). For $%
f,g,h\in C^{\infty }\left( M\right) ,$ we define

(i) 
\begin{equation*}
\begin{array}{ccl}
\Gamma \left( f,g\right) & = & \frac{1}{2}\Delta _{b}\left( fg\right)
-f\Delta _{b}g-g\Delta _{b}f.%
\end{array}%
\end{equation*}

(ii) 
\begin{equation*}
\begin{array}{ccl}
\Gamma _{2}\left( f,g\right) & = & \frac{1}{2}\left[ \Delta _{b}\Gamma
\left( f,g\right) -\Gamma \left( f,\Delta _{b}g\right) -\Gamma \left(
g,\Delta _{b}f\right) \right] .%
\end{array}%
\end{equation*}

(iii) 
\begin{equation*}
\begin{array}{ccl}
\Gamma ^{Z}\left( fg,h\right) & = & f\Gamma ^{Z}\left( g,h\right) +g\Gamma
^{Z}\left( f,h\right) .%
\end{array}%
\end{equation*}

(iv) 
\begin{equation*}
\begin{array}{ccl}
\Gamma _{2}^{Z}\left( f,g\right) & = & \frac{1}{2}\left[ \Delta _{b}\Gamma
^{Z}\left( f,g\right) -\Gamma ^{Z}\left( f,\Delta _{b}g\right) -\Gamma
^{Z}\left( g,\Delta _{b}f\right) \right] .%
\end{array}%
\end{equation*}
Here we denote $\Gamma \left( f\right) =\Gamma \left( f,f\right) ,\ \Gamma
_{2}\left( f\right) =\Gamma _{2}\left( f,f\right) ,\ \Gamma ^{Z}\left(
f\right) =\Gamma ^{Z}\left( f,f\right) $ and $\Gamma _{2}^{Z}\left( f\right)
=\Gamma _{2}^{Z}\left( f,f\right) .$ Note that in a complete strictly
pseudoconvex CR $(2n+1)$-manifold with vanishing torsion. One can have $%
\Gamma \left( f,f\right) =\left( \nabla _{b}f,\nabla _{b}f\right) $ and $%
\Gamma _{2}\left( f\right) =\| \nabla _{b}^{2}f\|^{2}+Ric\left( \nabla
_{b}f,\nabla _{b}f\right) +\frac{n}{2}\| \nabla _{T}\nabla _{b}f\|^{2}$ and $%
\Gamma^{Z}\left( f,g\right) =\left( \nabla _{T}f,\nabla _{T}g\right)$.

\begin{de}
We say that $(M,J,\theta )$ satisfies the generalized curvature-dimension
inequality $CD\left( \rho_{1},\rho_{2},\kappa ,d\right) $ with respect to $%
\Delta_{b}$ if there exist constants $\rho_{1}$ a real number, $\rho_{2}>0 $%
, $\kappa \geq 0$, and $d\geq 2$ such that the inequality 
\begin{equation*}
\Gamma_{2}( f) +\nu \Gamma_{2}^{Z}( f) \geq \frac{1}{ d}( \Delta_{b}f)^{2}+(
\rho _{1}-\frac{\kappa }{\nu } ) \Gamma ( f) +\rho _{2}\Gamma^{Z}( f)
\end{equation*}
holds for every $f\in C^{\infty }( M) $ and every $\nu>0.$
\end{de}

We define 
\begin{equation}
D:=d\left( 1+\frac{3\kappa }{2\rho _{2}}\right)  \label{30001}
\end{equation}%
and 
\begin{equation*}
\rho _{1}^{-}=\max \left( -\rho _{1},0\right) .
\end{equation*}

\begin{lem}
\label{l31} (i) (\cite[Theorem 4]{m}) Let $(M,J,\theta )$ be a complete
strictly pseudoconvex CR $(2n+1)$-manifold of vanishing torsion with 
\begin{equation*}
Ric\geq \rho _{1}.
\end{equation*}%
Then $M$ satisfies the generalized curvature-dimension inequality $CD\left(
\rho _{1},\frac{n}{2},1,2n\right) $ with $\rho _{2}=\frac{n}{2},\kappa =1$
and $d=2n$. Moreover for any given $R_{0}>0,$ there exists a constant $%
C\left( d,\kappa ,\rho _{2}\right) >0$ such that 
\begin{equation*}
\mu \left( B\left( x,R\right) \right) \leq C\left( d,\kappa ,\rho
_{2}\right) \frac{\exp \left( 2d\rho _{1}^{-}R_{0}^{2}\right) }{%
R_{0}^{D}p\left( x,x,R_{0}^{2}\right) }R^{D}\exp \left( 2d\rho
_{1}^{-}R^{2}\right)
\end{equation*}%
for every $x\in M$ and $R\geq R_{0}$. In particular if $M$ is a complete
strictly pseudoconvex CR $(2n+1)$-manifold of nonnegative Ricci curvature
and vanishing torsion, then there exists a constant $C_{1}>0$ such that 
\begin{equation}
\mu \left( B\left( x,R\right) \right) \leq \frac{C_{1}}{R_{0}^{D}p\left(
x,x,R_{0}^{2}\right) }R^{D}  \label{3-1}
\end{equation}%
for $R\geq R_{0}.$

(ii) (\cite{bg}) Let $(M,J,\theta )$ be a complete strictly pseudoconvex CR $%
(2n+1)$-manifold of nonnegative Ricci curvature and vanishing torsion. Then,
for any $\varepsilon >0,$ there exists a constant $C_{3}(d,\rho _{2},\kappa
, $ $\varepsilon )>0$ such that%
\begin{equation}
p\left( x,y,t\right) \leq \frac{C\left( d,\rho _{2},\kappa ,\varepsilon
\right) }{\mu \left( B\left( x,\sqrt{t}\right) \right) ^{\frac{1}{2}}\mu
\left( B\left( y,\sqrt{t}\right) \right) ^{\frac{1}{2}}}\exp \left( -\frac{%
d^{2}\left( x,y\right) }{\left( 4+\varepsilon \right) t}\right) .
\label{2006}
\end{equation}

(iii) (\cite{bbgm}) Let $(M,J,\theta )$ be a complete strictly pseudoconvex
CR $(2n+1)$-manifold of nonnegative Ricci curvature and vanishing torsion.
Then there exists a constant $C_{2}>0$ such that%
\begin{equation}
p\left( x,x,2R^{2}\right) \geq \frac{C_{2}}{\mu \left( B\left( x,R\right)
\right) }.  \label{3-2}
\end{equation}
\end{lem}

\begin{rem}
Let $(M,J,\theta )$ be a complete strictly pseudoconvex CR $(2n+1)$-manifold
of nonnegative Ricci curvature and vanishing torsion. (\ref{3-1}) and (\ref%
{3-2}) together imply the doubling property. That is%
\begin{equation}
\begin{array}{l}
\mu \left( B\left( x,R\right) \right) \leq \frac{C_{1}}{R_{0}^{D}p\left(
x,x,R_{0}^{2}\right) }R^{D}\leq C(\frac{R}{R_{0}})^{D}\mu \left( B\left( x,%
\frac{R_{0}}{\sqrt{2}}\right) \right) .%
\end{array}
\label{3-3}
\end{equation}%
By taking $R_{0}=\frac{R}{\sqrt{2}},$ then there exists a constant $C_{4}>0$
such that 
\begin{equation}
\begin{array}{l}
\mu \left( B\left( x,R\right) \right) \leq C_{4}\left( n,D\right) \mu \left(
B\left( x,\frac{R}{2}\right) \right) .%
\end{array}
\label{3-4}
\end{equation}
\end{rem}

Applying above Lemma \ref{l31}, we are able to prove the following moment
type estimate for those solution of form $P_{t}f$.

\begin{thm}
\label{prethm2-2} Let $(M,J,\theta )$ be a complete strictly pseudoconvex CR 
$(2n+1)$-manifold of nonnegative Ricci curvature and vanishing torsion.
Assume that $u$ is a solution of CR heat equation 
\begin{equation*}
\frac{\partial }{\partial t}u=\Delta _{b}u
\end{equation*}%
such that 
\begin{equation*}
u\left( x,t\right) =P_{t}f
\end{equation*}%
for a nonnegative bounded function $f$. Assume that for any $a>-D-2$ (where $%
D=2n+6$ is defined in \ref{30001}), we have 
\begin{equation*}
\frac{1}{V\left( r\right) }\int_{B_{x}\left( r\right) }f\left( y\right) d\mu
\left( y\right) \leq Ar^{a}
\end{equation*}%
for a constant $A>0$ and $r\geq R\geq 1$. Then there exists a constant $%
C\left( n,d\right) $ such that 
\begin{equation}
u\left( x,t\right) \leq C\left( n,d\right) At^{\frac{a}{2}}  \label{2013a}
\end{equation}%
for all $t\geq R^{2}$.
\end{thm}

\begin{proof}
Let $\delta =\frac{d\left( x,y\right) }{\sqrt{t}}.$ Thus 
\begin{equation}
\begin{array}{l}
B_{x}\left( \sqrt{t}\right) \subseteq B_{y}\left( \left( \delta +1\right) 
\sqrt{t}\right) .%
\end{array}
\label{2007}
\end{equation}%
It follow from (\ref{3-4}) and (\ref{2007}) that 
\begin{equation*}
\begin{array}{l}
V_{x}\left( \sqrt{t}\right) \leq V_{y}\left( \left( \delta +1\right) \sqrt{t}%
\right) \leq C\left( d,\kappa ,\rho _{2}\right) \left( \delta +1\right)
^{D}V_{y}\left( \sqrt{t}\right) .%
\end{array}%
\end{equation*}%
That is,%
\begin{equation}
\begin{array}{l}
\frac{V_{x}\left( \sqrt{t}\right) }{V_{y}\left( \sqrt{t}\right) }\leq
C\left( d,\kappa ,\rho _{2}\right) \left( \delta +1\right) ^{D}.%
\end{array}
\label{2009}
\end{equation}%
We can rewrite (\ref{2006}) as 
\begin{equation}
\begin{array}{lll}
p\left( x,y,t\right) & \leq & \frac{C\left( d,\rho _{2},\kappa ,\varepsilon
\right) }{\mu \left( B\left( x,\sqrt{t}\right) \right) ^{\frac{1}{2}}\mu
\left( B\left( y,\sqrt{t}\right) \right) ^{\frac{1}{2}}}\exp \left( -\frac{%
d^{2}\left( x,y\right) }{\left( 4+\varepsilon \right) t}\right) \\ 
& \leq & \frac{C\left( d,\rho _{2},\kappa ,\varepsilon \right) }{\mu \left(
B\left( x,\sqrt{t}\right) \right) }\left( \frac{\mu \left( B\left( x,\sqrt{t}%
\right) \right) }{\mu \left( B\left( y,\sqrt{t}\right) \right) }\right) ^{%
\frac{1}{2}}\exp \left( -\frac{d^{2}\left( x,y\right) }{\left( 4+\varepsilon
\right) t}\right) \\ 
& \leq & \frac{C\left( d,\rho _{2},\kappa ,\varepsilon \right) }{\mu \left(
B\left( x,\sqrt{t}\right) \right) }\exp \left( -\frac{d^{2}\left( x,y\right) 
}{\left( 4+\varepsilon \right) t}\right) .%
\end{array}
\label{2008}
\end{equation}%
Then, based on (\ref{2009}) and (\ref{2008}), Theorem \ref{prethm2-2}
follows from the proof of Theorem 3.1 in \cite{n3} in case of $u\left(
x,t\right) =P_{t}f$ for a nonnegative bounded function $f$. The use of
volume comparison can be replace by (\ref{3-3}).
\end{proof}

\section{CR Linear Trace Li-Yau-Hamilton Type Inequality}

In this section, we first relate the linear trace Li-Yau-Hamilton type
inequality of the CR Lichnerowicz-Laplacian heat equation to a monotonicity
formula of the heat solution. More precisely, let $\eta _{\alpha \bar{\beta}%
}\left( x,t\right) $ be a symmetric $\left( 1,1\right) $ tensor satisfying
the CR Lichnerowicz-Laplacian heat equation%
\begin{equation}
\frac{\partial }{\partial t}\eta _{\alpha \bar{\beta}}=\Delta _{b}\eta
_{\alpha \bar{\beta}}+2R_{\alpha \bar{\gamma}\mu \bar{\beta}}\eta _{\gamma 
\bar{\mu}}-(R_{\gamma \bar{\beta}}\eta _{\alpha \bar{\gamma}}+R_{\alpha \bar{%
\gamma}}\eta _{\gamma \bar{\beta}})  \label{4-1}
\end{equation}%
on $M\times \lbrack 0,T).$ As in the paper of \cite{ccf}, we define
following Harnack quantity

\begin{equation*}
Z\left( x,t\right) \left( V\right) :=k_{1}\left( \frac{1}{2}\left( \left( 
\text{div}\eta \right) _{\alpha ,\bar{\alpha}}+\left( \text{div}\eta \right)
_{\bar{\alpha},\alpha }\right) +\left( \text{div}\eta \right) _{\alpha }V_{%
\bar{\alpha}}+\left( \text{div}\eta \right) _{\alpha }V_{\bar{\alpha}}+V_{%
\bar{\alpha}}V_{\beta }\eta _{\alpha \bar{\beta}}\right) +\frac{H}{t}
\end{equation*}%
for any vector field $V\in T^{1,0}\left( M\right) ,\ H=$ $h^{\alpha 
\overline{\beta }}\eta _{\alpha \bar{\beta}}$ and $0<k_{1}\leq 8$. We proved

\begin{thm}
(\cite{ccf}) \label{thm2-1} Let $(M,J,\theta )$ be a complete strictly
pseudoconvex CR $(2n+1)$-manifold of nonnegative bisectional curvature and
vanishing torsion. Let $\eta _{\alpha \bar{\beta}}\left( x,t\right) $ be a
symmetric $\left( 1,1\right) $ tensor satisfying the CR
Lichnerowicz-Laplacian heat equation (\ref{4-1}) on $M\times \left(
0,T\right) $ with 
\begin{equation*}
\eta _{\alpha \bar{\beta}}\left( x,0\right) \geq 0,
\end{equation*}%
and%
\begin{equation*}
\nabla _{T}\eta \left( x,0\right) =0.
\end{equation*}%
Then 
\begin{equation*}
Z\left( x,t\right) \geq 0
\end{equation*}%
on $M\times \left( 0,T\right) $ for any $\left( 1,0\right) $ vector field $V$
and $0<k_{1}\leq 8$ if there exists constant $a>0$ such that%
\begin{eqnarray}
&&%
\begin{array}{l}
\int_{0}^{T}\int_{M}e^{-ar^{2}}\left \Vert \eta \left( x,t\right) \right
\Vert ^{2}d\mu dt<\infty ,%
\end{array}
\label{thm2-1-001} \\
&&%
\begin{array}{l}
\int_{0}^{T}\int_{M}e^{-ar^{2}}\left \Vert \nabla _{T}\eta \left( x,t\right)
\right \Vert ^{2}d\mu dt<\infty ,%
\end{array}
\label{thm2-1-002} \\
&&%
\begin{array}{l}
\int_{M}e^{-ar^{2}}\left \Vert \eta \left( x,0\right) \right \Vert d\mu
<\infty .%
\end{array}
\label{thm2-1-003}
\end{eqnarray}
\end{thm}

Let $\phi $ be a $\left( p,q\right) $-form. Define contraction operator $%
\Lambda :\Lambda ^{p,q}\rightarrow \Lambda ^{p-1,q-1}$ as follow%
\begin{equation*}
\begin{array}{l}
\left( \Lambda \phi \right) _{\alpha _{1}...\alpha _{p-1}\bar{\beta}_{1}...%
\bar{\beta}_{q-1}}=\frac{1}{\sqrt{-1}}\left( -1\right) ^{p-1}h^{\alpha \bar{%
\beta}}\phi _{\alpha \alpha _{1}...\alpha _{p-1}\bar{\beta}\bar{\beta}_{1}...%
\bar{\beta}_{q-1}.}%
\end{array}%
\end{equation*}

Then it is a straightforward computation, we have

\begin{lem}
(\cite{cct}) \label{lem3-1} Let $(M,J,\theta )$ be a strictly pseudoconvex
CR $(2n+1)$-manifold. We have the K\"{a}hler type identities\newline
(i) 
\begin{equation*}
\left[ \partial _{b},\Lambda \right] =-\sqrt{-1}\bar{\partial}_{b}^{\ast }%
\text{ \ \ \textrm{and} }\ \text{ }\left[ \bar{\partial}_{b},\Lambda \right]
=\sqrt{-1}\partial _{b}^{\ast }.
\end{equation*}%
(ii) 
\begin{equation*}
\lbrack \bar{\partial}_{b},\square _{b}]=2iT\bar{\partial}_{b}\text{ \ \ 
\textrm{and} \ \ }[\partial _{b},\square _{b}]=0.
\end{equation*}%
(iii) 
\begin{equation*}
\left[ \bar{\partial}_{b},\Delta _{H}\right] =-iT\bar{\partial}_{b}\text{ \
\ \textrm{and} \ \ }[\Lambda ,\Delta _{H}]=0.
\end{equation*}
\end{lem}

\begin{lem}
\label{lem3-2} Let $\phi $ be a nonnegative $\left( 1,1\right) $-form.
Define $Q\left( \phi ,V\right) $ as%
\begin{equation}
\begin{array}{l}
Q\left( \phi ,V,k_{2}\right) =k_{2}\left( \frac{1}{2\sqrt{-1}}\left( \bar{%
\partial}_{b}^{\ast }\partial _{b}^{\ast }-\partial _{b}^{\ast }\bar{\partial%
}_{b}^{\ast }\right) \phi +\frac{1}{\sqrt{-1}}\left( \bar{\partial}%
_{b}^{\ast }\phi \right) _{V}-\frac{1}{\sqrt{-1}}\left( \partial _{b}^{\ast
}\phi \right) _{V}+\phi _{V,\bar{V}}\right) +\frac{\Lambda \phi }{t}.%
\end{array}
\label{3002}
\end{equation}%
Then this is equivalent to 
\begin{equation*}
\begin{array}{l}
Q\left( \eta ,V,k_{2}\right) =k_{2}\left( \frac{1}{2}\left( \left( \text{div}%
\eta \right) _{\alpha ,\bar{\alpha}}+\left( \text{div}\eta \right) _{\bar{%
\alpha},\alpha }\right) +\left( \text{div}\eta \right) _{\alpha }V_{\bar{%
\alpha}}+\left( \text{div}\eta \right) _{\alpha }V_{\bar{\alpha}}+\eta
_{\alpha \bar{\beta}}V_{\alpha }V_{\bar{\beta}}\right) +\frac{H}{t}%
\end{array}%
\end{equation*}%
for a symmetric $\left( 1,1\right) $-tensor $\eta _{\alpha \bar{\beta}}:=%
\frac{1}{\sqrt{-1}}\phi _{\alpha \bar{\beta}}$. In particular, by taking $%
V=0 $, and $k_{2}=2$, we have%
\begin{equation}
\begin{array}{l}
Q\left( \phi ,V\right) =-\Delta _{H}\Lambda \phi +\left( \bar{\partial}%
_{b}^{\ast }\Lambda \bar{\partial}_{b}+conj\right) \phi +\frac{u}{t}%
\end{array}
\label{3001}
\end{equation}%
for $u=\Lambda \phi .$
\end{lem}

\begin{proof}
As in \cite{ccf}, we have the formula for a $\left( p,q+1\right) $-form $%
\psi $ 
\begin{equation*}
\begin{array}{l}
\left( \bar{\partial}_{b}^{\ast }\psi \right) _{\alpha _{1}...\alpha _{p}%
\bar{\beta}_{1}...\bar{\beta}_{q}}=\left( -1\right) ^{p}\frac{1}{q+1}%
\sum_{i=1}^{q+1}\left( -1\right) ^{i}\nabla _{\mu }\psi _{\alpha
_{1}...\alpha _{p}\bar{\beta}_{1}...\bar{\beta}_{i-1}\bar{\mu}\bar{\beta}%
_{i}...\bar{\beta}_{q}}%
\end{array}%
\end{equation*}%
and a $\left( p+1,q\right) $-form $\varphi $ 
\begin{equation*}
\begin{array}{l}
\left( \partial _{b}^{\ast }\varphi \right) _{\alpha _{1}...\alpha _{p}\bar{%
\beta}_{1}...\bar{\beta}_{q}}=\left( -1\right) \frac{1}{p+1}\nabla _{\bar{\mu%
}}\varphi _{\mu \alpha _{1}...\alpha _{p}\bar{\beta}_{1}...\bar{\beta}_{q}}.%
\end{array}%
\end{equation*}%
Thus for a $\left( 1,1\right) $-form $\phi ,$ we have 
\begin{equation*}
\left( \partial _{b}^{\ast }\phi \right) _{\bar{\gamma}}=-\nabla _{\bar{\mu}%
}\phi _{\mu \bar{\gamma}}
\end{equation*}%
and 
\begin{equation*}
\bar{\partial}_{b}^{\ast }\partial _{b}^{\ast }\phi =\nabla _{\gamma }\left(
\nabla _{\bar{\mu}}\phi _{\mu \bar{\gamma}}\right) .
\end{equation*}%
Then the first term of (\ref{3002}) become%
\begin{equation*}
\begin{array}{lll}
\frac{1}{2\sqrt{-1}}\left( \bar{\partial}_{b}^{\ast }\partial _{b}^{\ast
}-\partial _{b}^{\ast }\bar{\partial}_{b}^{\ast }\right) \phi & = & \frac{1}{%
2\sqrt{-1}}\bar{\partial}_{b}^{\ast }\partial _{b}^{\ast }\phi +conj. \\ 
& = & \frac{1}{2\sqrt{-1}}\nabla _{\gamma }\left( \nabla _{\bar{\mu}}\phi
_{\mu \bar{\gamma}}\right) +conj. \\ 
& = & \frac{1}{2}\left( \left( \text{div}\eta \right) _{\alpha ,\bar{\alpha}%
}+conj.\right)%
\end{array}%
\end{equation*}%
We are done. On the other hand, taking $V=0$ and $k_{2}=2$, by lemma \ref%
{lem3-1} we have%
\begin{equation*}
\begin{array}{lll}
\frac{2}{2\sqrt{-1}}\left( \bar{\partial}_{b}^{\ast }\partial _{b}^{\ast
}-\partial _{b}^{\ast }\bar{\partial}_{b}^{\ast }\right) \phi & = & \bar{%
\partial}_{b}^{\ast }[\bar{\partial}_{b},\Lambda ]\phi -\partial _{b}^{\ast }%
\left[ \partial _{b},\Lambda \right] \phi \\ 
& = & -\bar{\partial}_{b}^{\ast }\bar{\partial}_{b}\Lambda \phi -\partial
_{b}^{\ast }\partial _{b}\Lambda \phi +\bar{\partial}_{b}^{\ast }\Lambda 
\bar{\partial}_{b}\phi +\partial _{b}^{\ast }\Lambda \partial _{b}\phi \\ 
& = & -\Delta _{H}\Lambda \phi +\bar{\partial}_{b}^{\ast }\Lambda \bar{%
\partial}_{b}\phi +\partial _{b}^{\ast }\Lambda \partial _{b}\phi .%
\end{array}%
\end{equation*}

Here we use the fact that $\bar{\partial}_{b}^{\ast }f=\partial _{b}^{\ast
}f=0$ for any scalar function. Then formula (\ref{3001}) follows.
\end{proof}

\bigskip

\begin{rem}
The regularity of the heat solution in Proposition \ref{prop2-1} and the
following Lemma is used to prove the "mix-term" $\left( \bar{\partial}%
_{b}^{\ast }\Lambda \bar{\partial}_{b}+conj\right) \phi $ in (\ref{3001})
vanishing as in (\ref{4006}) and (\ref{4006b}) which is the key step in the
proof of our main theorem.
\end{rem}

\bigskip

\begin{lem}
\label{lem3-3} Let $(M,J,\theta )$ be a complete strictly pseudoconvex CR $%
(2n+1)$-manifold with nonnegative bisectional curvature and vanishing
torsion. Let $\phi $ be a solution of the CR Hodge-Laplace heat equation (%
\ref{70}). Then $\left \Vert \Lambda \bar{\partial}_{b}\phi \right \Vert $
satisfies%
\begin{equation*}
\begin{array}{ll}
\left( \frac{\partial }{\partial t}-\Delta _{b}\right) \left \Vert \Lambda 
\bar{\partial}_{b}\phi \right \Vert & \leq ||\Lambda T\bar{\partial}_{b}\phi
||.%
\end{array}%
\end{equation*}
\end{lem}

\begin{proof}
We have the formula for a $\left( p,q\right) $-form $\psi $%
\begin{equation*}
\begin{array}{l}
\left( \bar{\partial}_{b}\psi \right) _{\alpha _{1}...\alpha _{p}\bar{\beta}%
_{1}...\bar{\beta}_{q+1}}=\left( -1\right) ^{p}\sum_{i=1}^{q+1}\left(
-1\right) ^{i-1}\nabla _{\bar{\beta}_{i}}\psi _{\alpha _{1}...\alpha _{p}%
\bar{\beta}_{1}...\bar{\beta}_{i-1}\bar{\beta}_{i+1}...\bar{\beta}_{q+1}}.%
\end{array}%
\end{equation*}%
So that 
\begin{equation*}
\left( \bar{\partial}_{b}\phi \right) _{\alpha \bar{\beta}\bar{\gamma}%
}=-\nabla _{\bar{\beta}}\phi _{\alpha \bar{\gamma}}+\nabla _{\bar{\gamma}%
}\phi _{\alpha \bar{\beta}}
\end{equation*}%
and 
\begin{equation*}
\begin{array}{lll}
\left( \Lambda \bar{\partial}_{b}\phi \right) _{\bar{\gamma}} & = & 
ih^{\alpha \bar{\beta}}\nabla _{\bar{\beta}}\phi _{\alpha \bar{\gamma}%
}-ih^{\alpha \bar{\beta}}\nabla _{\bar{\gamma}}\phi _{\alpha \bar{\beta}} \\ 
& = & h^{\alpha \bar{\beta}}\nabla _{\bar{\beta}}\eta _{\alpha \bar{\gamma}%
}-h^{\alpha \bar{\beta}}\nabla _{\bar{\gamma}}\eta _{\alpha \bar{\beta}} \\ 
& = & \left( \text{div}\eta \right) _{\bar{\gamma}}-\nabla _{\bar{\gamma}}u.%
\end{array}%
\end{equation*}

Note that $\Lambda \bar{\partial}_{b}\phi $ satisfies the CR Hodge Laplace
heat equation, i.e., 
\begin{equation*}
\begin{array}{lll}
\left( \frac{\partial }{\partial t}+\Delta _{H}\right) \Lambda \bar{\partial}%
_{b}\phi & = & -\Lambda \bar{\partial}_{b}\Delta _{H}\phi +\Delta
_{H}\Lambda \bar{\partial}_{b}\phi \\ 
& = & -\Lambda \Delta _{H}\bar{\partial}_{b}\phi +\Delta _{H}\Lambda \bar{%
\partial}_{b}\phi +i\Lambda T\bar{\partial}_{b}\phi \\ 
& = & [\Delta _{H},\Lambda ]\bar{\partial}_{b}\phi +i\Lambda T\bar{\partial}%
_{b}\phi \\ 
& = & i\Lambda T\bar{\partial}_{b}\phi .%
\end{array}%
\end{equation*}%
Hence we have%
\begin{equation*}
\begin{array}{ll}
\left( \frac{\partial }{\partial t}-\Delta _{b}\right) \sqrt{\left \Vert
\Lambda \bar{\partial}_{b}\phi \right \Vert ^{2}} & =\frac{\Lambda \bar{%
\partial}_{b}\phi }{\left \Vert \Lambda \bar{\partial}_{b}\phi \right \Vert }%
\cdot \left( -\Delta _{H}\Lambda \bar{\partial}_{b}\phi -\Delta _{b}\Lambda 
\bar{\partial}_{b}\phi \right) +\frac{\Lambda \bar{\partial}_{b}\phi }{\left
\Vert \Lambda \bar{\partial}_{b}\phi \right \Vert }\cdot i\Lambda T\bar{%
\partial}_{b}\phi \\ 
& =-\frac{1}{\left \Vert \Lambda \bar{\partial}_{b}\phi \right \Vert }%
R_{\alpha \bar{\beta}}\left( \Lambda \bar{\partial}_{b}\phi \right) _{\alpha
}\overline{\left( \Lambda \bar{\partial}_{b}\phi \right) }_{\bar{\beta}}+%
\frac{\Lambda \bar{\partial}_{b}\phi }{\left \Vert \Lambda \bar{\partial}%
_{b}\phi \right \Vert }\cdot i\Lambda T\bar{\partial}_{b}\phi%
\end{array}%
\end{equation*}%
where in second line we use formula (3.1) of \cite{ccf} for $\left(
1,0\right) $-form $\Lambda \bar{\partial}_{b}\phi $.
\end{proof}

\bigskip

Before going any further for the proof of our main theorem, we need two more
lemmas.

\begin{lem}
(\cite{nt2}) \label{lem2-2} Let $f\geq 0$ be a function on a complete
noncompact Riemannian manifold $M^{m}$ with%
\begin{equation*}
R_{ij}\geq -\left( m-1\right) K
\end{equation*}%
for some $K\geq 0.$ Let 
\begin{equation*}
u\left( x,t\right) :=\int_{M}H\left( x,y,t\right) f\left( y\right) dy.
\end{equation*}%
Assume that $u$ is defined on $M\times \lbrack 0,T]$ for some $T>0$ and that
for $0<t\leq T,$%
\begin{equation}
\lim_{r\longrightarrow \infty }\exp \left( -\frac{r^{2}}{20t}\right)
\int_{B_{o}\left( r\right) }f=0.  \label{44}
\end{equation}%
and $p\geq 1,$%
\begin{equation*}
\begin{array}{ll}
& \frac{1}{V_{o}\left( r\right) }\int_{B_{o}\left( r\right) }u^{p}dx \\ 
\leq & C_{m,p}\left[ \frac{1}{V_{o}\left( 4r\right) }\int_{B_{o}\left(
4r\right) }f^{p}dx+\left( C_{2}\left( K,t\right) \int_{4r}^{\infty }\left( 
\frac{s}{\sqrt{t}}+\frac{s^{2}}{t}\right) \exp \left( -\frac{s^{2}}{40t}%
\right) \frac{1}{V_{o}\left( s\right) }\int_{B_{o}\left( s\right) }fd\left( 
\frac{s^{2}}{t}\right) \right) ^{p}\right]%
\end{array}%
\end{equation*}%
where $C_{2}\left( K,t\right) =C_{m}te^{C_{m}Kt}$ and $C_{m}$ is constant
only depend on dimension $M.$
\end{lem}

\begin{lem}
(\cite{li}) \label{mp1} Let $(M,J,\theta )$ be a complete strictly
pseudoconvex CR $(2n+1)$-manifold and $f\left( x,t\right) $ be the
subsolution of the heat equation satisfying%
\begin{equation*}
\begin{array}{l}
\left( \frac{\partial }{\partial t}-\Delta _{b}\right) f\left( x,t\right)
\leq 0\text{ on }M\times \lbrack 0,T)%
\end{array}%
\end{equation*}%
with $f\left( x,0\right) \leq 0$ on $M$. Then $f\left( x,t\right) \leq 0$
for all $t<T$ if there exists $a>0$ such that 
\begin{equation*}
\begin{array}{l}
\int_{0}^{T}\int_{M}f^{2}\left( x,t\right) e^{-ar^{2}}d\mu \left( x\right)
dt<\infty .%
\end{array}%
\end{equation*}
\end{lem}

\section{Proof of CR Optimal Gap Theorem}

In this section, by using the CR moment type estimate (Theorem \ref%
{prethm2-2}) and the linear trace LYH inequality (Theorem \ref{thm2-1}), we
are able to prove the CR optimal gap theorem.

\textbf{Proof of the main theorem:}

\begin{proof}
Here is the main idea : In the following we first use proposition \ref%
{prop2-1} to construct $\eta _{\mu }$ on exhaustion domain $\Omega _{\mu }$.
Schauder estimates provide the convergence of $\eta _{\mu }$ (\textbf{Step 1}%
) to a unique solution $\eta $. Define $u:=tr_{h}\eta $ and $u$ is a
solution of sublaplacian heat equation with initial condition $S\left(
y\right) $. By uniqueness theorem (\cite{d}) of the nonnegative heat
solution we have $u^{\left( i\right) }\rightarrow u$. Now this allows us in
one hand using trace linear Harnack estimate on $tr_{h}\eta $ to obtain
monotonicity formula 
\begin{equation}
\left( tu\right) _{t}\geq 0  \label{2013}
\end{equation}%
which apply to every nonnegative heat solution and on the other hand using
moment type estimate (which only apply to heat solution with $P_{t}f$ type
and $f$ is bounded) on $u^{\left( i\right) }:=P_{t}\rho ^{\left( i\right) }S$
to obtain that 
\begin{equation*}
u^{\left( i\right) }=o\left( t^{-1}\right) .
\end{equation*}%
Hence as well as $u$. Combing these results, the initial condition are
forced to be zero and the gap theorem holds.

Note that we derived the monotonicity property (\ref{2013}) by lemma \ref%
{lem3-2}, \ref{lem3-3}, and the vanishing of mixed term in LYH quantity (\ref%
{3001}). The condition (\ref{2001}) is applied while we use Theorem \ref%
{prethm2-2} for $a=-2$ to obtain%
\begin{equation*}
u=o\left( t^{-1}\right) .
\end{equation*}

Now we split the detail proof into two steps :

\ \ (i) \textbf{Step 1 : Convergence of }$\eta _{\mu }^{\left( i\right) }$%
\textbf{: }\ Let $\Omega _{\mu }$ be an sweetsop exhaustion domains, $\rho
^{\left( i\right) }$ be a cut-off function support in $B\left( 2R_{i}\right) 
$ such that $0\leq \rho ^{\left( i\right) }\leq 1$, $\rho ^{\left( i\right)
}=1$ in $B\left( R_{i}\right) ,$ $\left \Vert \nabla _{b}^{m_{1}}\nabla
_{T}^{m_{2}}\rho ^{\left( i\right) }\right \Vert \leq $ $\frac{C}{R_{i}}$
for $m_{1},m_{2}=0,1,2,\ m_{1}+m_{2}\geq 1$ and some constant $C$. Note that
for each $i$, there exists $N_{i}$ such that for $\mu \geq N_{i},$ $B\left(
2R_{i}\right) \subset \Omega _{\mu }$. Let $\eta _{\mu }^{\left( i\right) }$
be the solution as in Proposition \ref{prop2-1} on $\Omega _{\mu }$ for any $%
\mu \geq N_{i}$ with initial condition $\rho ^{\left( i\right) }Ric$. Now we
define 
\begin{equation*}
\begin{array}{l}
u^{\left( i\right) }\left( x,t\right) :=\int_{M}p\left( x,y,t\right) \rho
^{\left( i\right) }S\left( y\right) d\mu \left( y\right) \text{.}%
\end{array}%
\end{equation*}%
Then $u^{\left( i\right) }\left( x,t\right) $ satisfies 
\begin{equation}
\begin{array}{l}
\frac{\partial }{\partial t}u^{\left( i\right) }\left( x,t\right) -\Delta
_{\varepsilon }u^{\left( i\right) }\left( x,t\right) =-\varepsilon
^{2}u_{00}^{\left( i\right) }\left( x,t\right) ,%
\end{array}
\label{4010}
\end{equation}%
where $\Delta _{\varepsilon }=\Delta _{b}+\varepsilon ^{2}T^{2}$ is
Riemannian Laplacian with respect to the adapted metric $h_{\varepsilon
}:=h+\varepsilon ^{-2}\theta ^{2}$. Moreover, proposition \ref{prop2-2}
imply $\eta _{\mu }^{\left( i\right) }\left( x,t\right) $ is nonnegative and 
\begin{equation}
\begin{array}{l}
\left \Vert \eta _{\mu }^{\left( i\right) }\left( x,t\right) \right \Vert
\leq tr_{h}\eta _{\mu }^{\left( i\right) }\left( x,t\right) \leq u^{\left(
i\right) }\left( x,t\right) ,%
\end{array}
\label{4017}
\end{equation}%
for all $\mu \geq N_{i}$. Now we estimate $u_{00}^{\left( i\right) }\left(
x,t\right) $ first. Since $u^{\left( i\right) }\left( x,t\right) $ is a
solution of sub-Laplacian heat equation, we have 
\begin{equation*}
\begin{array}{l}
\frac{\partial }{\partial t}u_{00}^{\left( i\right) }\left( x,t\right)
-\Delta _{b}u_{00}^{\left( i\right) }\left( x,t\right) =0%
\end{array}%
\end{equation*}%
due to vanishing torsion. We define $l^{\left( i\right) }\left( x,t\right)
=\left \vert u_{00}^{\left( i\right) }\left( x,t\right) \right \vert $, and
observe that it is a subsolution of heat equation with initial condition
satisfying the followings 
\begin{equation*}
\begin{array}{ll}
& \left \vert l^{\left( i\right) }\left( x,t\right) \left( x,0\right) \right
\vert \\ 
= & \left \vert \nabla _{T}\nabla _{T}\rho ^{\left( i\right) }S\left(
y\right) \right \vert \\ 
\leq & \frac{C}{R_{i}}\chi _{B_{2R_{i}\backslash R_{i}}}S\left( y\right) ,%
\end{array}%
\end{equation*}%
where $\chi _{B_{2R_{i}\backslash R_{i}}}\left( y\right) $ is a function
with $1$ in annulus $B\left( 2R_{i}\right) \backslash B\left( R_{i}\right) $
and zero elsewhere. By maximum principle $l^{\left( i\right) }\left(
x,t\right) $ is controlled by a sub-Laplacian heat solution.

Next we define 
\begin{equation*}
g\left( x,t\right) :=\int_{M}p\left( x,y,t\right) \frac{C}{R_{i}}\chi
_{B_{2R_{i}\backslash R_{i}}}S\left( y\right) dy.
\end{equation*}
By moment type estimate 
\begin{equation}
\begin{array}{l}
g\left( x,t\right) =\frac{1}{R_{i}}o\left( t^{-1}\right) ,%
\end{array}
\label{4013}
\end{equation}%
where the particular coefficient in $o\left( t^{-1}\right) $ does not depend
on $i$. To summarize, we have 
\begin{equation}
\begin{array}{l}
\left \vert u_{00}^{\left( i\right) }\left( x,t\right) \right \vert
=l^{\left( i\right) }\left( x,t\right) \leq g\left( x,t\right) =\frac{1}{%
R_{i}^{2}}o\left( t^{-1}\right) .%
\end{array}
\label{4011}
\end{equation}

We return to equation (\ref{4010}). Now we restricted on $B\left( r\right)
\times \left[ \epsilon ,T\right] $ and try to obtain estimate not depend on
index $i$. Now we define 
\begin{equation*}
\begin{array}{l}
L^{\left( i\right) }\left( x,t\right) =u^{\left( i\right) }\left( x,t\right)
+\varepsilon ^{2}e^{T-t}\sup \limits_{B\left( r\right) \times \left[
\epsilon ,T\right] }g\left( x,t\right)%
\end{array}%
\end{equation*}%
so that $L^{\left( i\right) }\left( x,t\right) $ satisfy%
\begin{equation}
\begin{array}{l}
\frac{\partial }{\partial t}L^{\left( i\right) }\left( x,t\right) -\Delta
_{\varepsilon }L^{\left( i\right) }\left( x,t\right) \leq 0.%
\end{array}
\label{4002}
\end{equation}

Applying mean value theorem (Theorem 1.2 in \cite{lt}) to function $%
L^{\left( i\right) }\left( x,t\right) ,$ we have%
\begin{equation*}
\begin{array}{lll}
\sup \limits_{B_{\varepsilon }\left( \left( 1-\delta \right) r\right) \times
\lbrack \epsilon ,T]}L^{\left( i\right) } & \leq & C_{16}\left \{ \frac{1}{%
\left( \delta r\right) ^{2n+3}}\frac{V_{\varepsilon }\left( \frac{2}{%
\varepsilon ^{2}},2r\right) }{V_{\varepsilon }\left( r\right) }\left( r\frac{%
\sqrt{2}}{\varepsilon }\coth \left( r\frac{\sqrt{2}}{\varepsilon }\right)
+1\right) \exp \left( C_{17}\frac{2}{\varepsilon ^{2}}T\right) \right \} \\ 
&  & \times \int_{\epsilon }^{T}ds\int_{B_{\varepsilon }\left( r\right)
}L^{\left( i\right) }\left( y,s\right) d\mu _{\varepsilon }\left( y\right)
+\left( 1+\varepsilon _{1}\right) \sup \limits_{B_{\varepsilon }\left(
r\right) }L^{\left( i\right) }\left( \cdot ,\epsilon \right) .%
\end{array}%
\end{equation*}%
Let $B_{\varepsilon }\left( r\right) ,d\mu _{\varepsilon }\left( y\right) $
denote the ball with radius $r$ and volume element which is respected to
metric $h_{\varepsilon }$. The above inequality also means%
\begin{equation}
\begin{array}{lll}
\sup \limits_{B_{\varepsilon }\left( \left( 1-\delta \right) r\right) \times
\lbrack \epsilon ,T]}u^{\left( i\right) } & \leq & C_{16}\left \{ \frac{1}{%
\left( \delta r\right) ^{2n+3}}\frac{V_{\varepsilon }\left( \frac{2}{%
\varepsilon ^{2}},2r\right) }{V_{\varepsilon }\left( r\right) }\left( r\frac{%
\sqrt{2}}{\varepsilon }\coth \left( r\frac{\sqrt{2}}{\varepsilon }\right)
+1\right) \exp \left( C_{17}\frac{2}{\varepsilon ^{2}}T\right) \right \} \\ 
&  & \times \int_{0}^{T}ds\int_{B_{\varepsilon }\left( r\right) }L^{\left(
i\right) }\left( y,s\right) d\mu _{\varepsilon }\left( y\right) \\ 
&  & +\left( 1+\varepsilon _{1}\right) \sup \limits_{B_{\varepsilon }\left(
r\right) }u^{\left( i\right) }\left( \cdot ,\epsilon \right) +\left(
1+\varepsilon _{1}\right) \varepsilon ^{2}e^{T-\epsilon }\sup
\limits_{B_{\varepsilon }\left( r\right) }g\left( x,\epsilon \right) .%
\end{array}
\label{40014}
\end{equation}

We only need to estimate the first term of (\ref{40014}) below, since the
other terms are bounded. We define 
\begin{equation*}
\begin{array}{l}
L_{\varepsilon }^{\left( i\right) }\left( y,s\right)
:=\int_{M}H_{\varepsilon }\left( x,y,t\right) \left \Vert \rho ^{\left(
i\right) }S\right \Vert \left( y\right) d\mu _{\varepsilon }\left( y\right)%
\end{array}%
\end{equation*}%
and again we have 
\begin{equation}
\begin{array}{l}
L^{\left( i\right) }\left( y,s\right) \leq L_{\varepsilon }^{\left( i\right)
}\left( y,s\right) +\varepsilon ^{2}e^{T}\sup_{B\left( r\right) \times \left[
\epsilon ,T\right] }g\left( x,t\right) \leq L_{\varepsilon }^{\left(
i\right) }\left( y,s\right) +\varepsilon ^{2}e^{T}\frac{o\left(
t^{-1}\right) }{R_{i}^{2}}.%
\end{array}
\label{40017}
\end{equation}%
Now the first term of (\ref{40014}) is estimated by using (\ref{40017}) and
Lemma \ref{lem2-2} as following%
\begin{equation}
\begin{array}{ll}
& \frac{1}{V_{o,\varepsilon }\left( r\right) }\int_{B_{o}\left( r\right)
}L_{\varepsilon }^{\left( i\right) }\left( y,t\right) d\mu _{\varepsilon
}\left( y\right) \\ 
\leq & C_{m,1}\frac{1}{V_{o,\varepsilon }\left( 4r\right) }%
\int_{B_{o,\varepsilon }\left( 4r\right) }\left \Vert \rho ^{\left( i\right)
}S\right \Vert \left( y\right) d\mu _{\varepsilon }\left( y\right) \\ 
+ & C_{m}e^{C_{m}\frac{1}{\varepsilon ^{2}}t}\int_{4r}^{\infty }\left( \frac{%
s}{\sqrt{t}}+\frac{s^{2}}{t}\right) \exp \left( -\frac{s^{2}}{40t}\right) 
\frac{1}{V_{o,\varepsilon }\left( s\right) }\int_{B_{o,_{\varepsilon
}}\left( s\right) }\left \Vert \rho ^{\left( i\right) }S\right \Vert \left(
y\right) d\mu _{\varepsilon }\left( y\right) d\left( \frac{s^{2}}{t}\right) .%
\end{array}
\label{4014}
\end{equation}

The integral $\int_{B_{o,\varepsilon }\left( 4r\right) }\left \Vert \rho
^{\left( i\right) }S\right \Vert \left( y\right) d\mu _{\varepsilon }\left(
y\right) $ inside both terms in (\ref{4014}) are estimated by assumption (%
\ref{2001}) and is controlled by quantity that not depend on $i$. Hence (\ref%
{40014}) and (\ref{4014}) imply 
\begin{equation}
\begin{array}{l}
\sup \limits_{B_{\varepsilon }\left( \left( 1-\delta \right) r\right) \times
\lbrack \epsilon ,T]}u^{\left( i\right) }\leq C\left( \varepsilon
,r,T,n,\rho ^{\left( i\right) }S\right)%
\end{array}
\label{40019}
\end{equation}

and 
\begin{equation}
\begin{array}{l}
\max_{B_{\varepsilon }\left( r\right) \times \lbrack \epsilon ,T]}tr_{h}\eta
_{\mu }^{\left( i\right) }\left( x,t\right) \leq C\left( \varepsilon
,r,T,n,\rho ^{\left( i\right) }S\right) \text{.}%
\end{array}
\label{40018}
\end{equation}%
Now the interior Schauder estimate can be applied to extract a convergent
subsequence $\eta _{\mu _{k}}^{\left( i\right) }\rightarrow \eta ^{\left(
i\right) }$ that satisfies the CR Lichnerowiz-subLaplacian heat equation on $%
[0,T]$. Note that $tr_{h}\eta ^{\left( i\right) }\left( x,0\right)
=u^{\left( i\right) }\left( x,0\right) $, and by uniqueness of bounded
sub-Laplacian heat solution (from lemma \ref{mp1}) we actually have 
\begin{equation*}
tr_{h}\eta ^{\left( i\right) }\left( x,t\right) =u^{\left( i\right) }\left(
x,t\right) .
\end{equation*}

By (\ref{4010}), (\ref{4011}), (\ref{40019}) and Schauder estimates, there
is a subsequence $u^{\left( i_{j}\right) }\rightarrow u$ and $\eta ^{\left(
i_{j}\right) }\rightarrow \eta $ in any fixed compact subset with an
arbitrary chosen H\"{o}lder norm (by choosing $\beta _{0}$ large for
sweetsop domain, see appendix). Note in (\ref{4011}) as $i$ goes to infinity
we can conclude $\nabla _{T}\nabla _{T}u\left( x,t\right) =0$ and similarly $%
\nabla _{T}u\left( x,t\right) =0$ and $\nabla _{T}\eta \left( x,t\right) =0$
by using that $\left \Vert \eta _{0}^{\left( i\right) }\right \Vert $ is a
subsolution of sub-Laplacian heat equation as follows%
\begin{equation*}
\begin{array}{l}
\left( \frac{\partial }{\partial t}-\Delta _{b}\right) \left \Vert \eta
_{0}^{\left( i\right) }\right \Vert =\frac{1}{\left \Vert \eta _{0}^{\left(
i\right) }\right \Vert }\left( 2R_{\alpha \bar{\gamma}\mu \bar{\beta}}\eta
_{0\gamma \bar{\mu}}-(R_{\gamma \bar{\beta}}\eta _{0\alpha \bar{\gamma}%
}+R_{\alpha \bar{\gamma}}\eta _{0\gamma \bar{\beta}})\eta _{0\zeta \bar{\xi}%
}h_{\beta \bar{\zeta}}h_{\xi \bar{\alpha}}\right) \leq 0.%
\end{array}%
\end{equation*}%
Here we use the facts that bisectional curvature is nonnegative and
vanishing torsion. Moreover, requirement for applying maximum principle is
garanteed by similar argument as (\ref{4017}), we have 
\begin{equation}
\begin{array}{l}
\left \Vert \eta _{0}^{\left( i\right) }\right \Vert \left( x,t\right) \leq
C\int p\left( x,y,t\right) \left \vert \nabla \rho \right \vert S\left(
y\right) d\mu \left( y\right) \leq \frac{C}{R_{i}}o\left( t^{-1}\right)%
\end{array}%
.  \label{4016}
\end{equation}%
As $i$ goes to infinity, $\eta _{0}=0$. However, by now we do not know yet
through the subsequence the two functions $tr_{h}\eta \left( x,t\right) $
and $u\left( x,t\right) $ are the same even they have the same initial
condition. One regards both $u\left( x,t\right) $ and $tr_{h}\eta \left(
x,t\right) $ as solutions of Laplacian heat equations associated to adapted
metric (due to $\nabla _{T}u\left( x,t\right) =\nabla _{T}tr_{h}\eta \left(
x,t\right) =0$), and the manifold are seen as Riemannian manifold with
Riemannian curvature bounded below by $-\frac{1}{\varepsilon ^{2}}$ (Theorem
4.9 in \cite{cc1}). Now by the uniqueness of nonnegative Laplacian heat
solution (\cite{d}) on complete manifold with Riemannian Ricci curvature
bounded below, we can conclude that%
\begin{equation*}
\begin{array}{l}
u\left( x,t\right) =tr_{h}\eta \left( x,t\right) .%
\end{array}%
\end{equation*}

Note that $u$ is the unique sub-Laplacian heat solution with $\nabla
_{T}u\left( x,t\right) =0$, and since any such $u$ we can find a sequence of 
$u^{\left( i\right) }$ that satisfy moment type estimates converge to $u$.
Hence $u$ satisfy the moment type estimate.

\ \ (ii) \textbf{Step 2 : Monotonicity of }$tu$ \textbf{: }By our
assumptions on $Ric$, and the upper bound of $\eta \left( x,t\right) $ by $%
u\left( x,t\right) =o\left( t^{-1}\right) $, (\ref{thm2-1-001}), \ (\ref%
{thm2-1-002}) and (\ref{thm2-1-003}) in Theorem \ref{thm2-1} are satisfied.
Hence by Lemma \ref{lem3-2} and (\ref{3001}), $tr_{h}\eta $ satisfy%
\begin{equation}
\begin{array}{l}
u_{t}+\left( \bar{\partial}_{b}^{\ast }\Lambda \bar{\partial}%
_{b}+conj\right) \phi +\frac{u}{t}\geq 0.%
\end{array}
\label{4006}
\end{equation}%
In the following we are going to prove the mixed terms $\left( \bar{\partial}%
_{b}^{\ast }\Lambda \bar{\partial}_{b}+conj\right) \phi $ of (\ref{4006})
vanishing so the monotonicity 
\begin{equation}
\begin{array}{l}
\left( tu\right) _{t}\geq 0%
\end{array}
\label{4006b}
\end{equation}%
follows. Hence 
\begin{equation*}
\begin{array}{l}
tu(x,t)\equiv 0,%
\end{array}%
\end{equation*}%
for any $x$ and $t>0$. The flatness then follows from $u(x,0)\equiv 0$.

In fact, we first define $\sigma ^{\left( i\right) }:=\Lambda \bar{\partial}%
_{b}\eta ^{\left( i\right) }$ (note $\eta ^{\left( i\right) }=\frac{1}{\sqrt{%
-1}}\phi ^{\left( i\right) }$). Then direct calculation shows that 
\begin{equation*}
\begin{array}{l}
\left( \frac{\partial }{\partial t}-\Delta _{b}\right) \left \Vert \eta
_{\mu _{k}}^{\left( i\right) }\right \Vert ^{2}\leq -2\left \Vert \nabla
\eta _{\mu _{k}}^{\left( i\right) }\right \Vert ^{2}\text{.}%
\end{array}%
\end{equation*}%
We integrate on both sides over $\Omega _{\mu _{k}}$ and apply Dirichlet
condition (using boundary regularity in Proposition \ref{prop2-1}). After
taking $\mu _{k}\rightarrow \infty $, we have 
\begin{equation}
\begin{array}{l}
2\int_{0}^{t}\int_{M}\left \Vert \nabla _{b}\eta ^{\left( i\right) }\right
\Vert ^{2}\left( x,s\right) d\mu ds\leq \int_{M}\left \Vert \eta ^{\left(
i\right) }\left( x,0\right) \right \Vert ^{2}d\mu =\int_{M}\left \Vert \rho
^{\left( i\right) }Ric\right \Vert ^{2}d\mu .%
\end{array}
\label{4012}
\end{equation}

Due to $\left \Vert \sigma ^{\left( i\right) }\right \Vert \left( x,t\right)
\leq \left \Vert \nabla _{b}\eta ^{\left( i\right) }\right \Vert \left(
x,t\right) $, (\ref{4012}) and assumption (\ref{2001}) we have for some $%
a^{\prime }>0,$ 
\begin{equation}
\begin{array}{l}
\int_{0}^{t}\int_{M}e^{-a^{\prime }r^{2}}\left \Vert \sigma ^{\left(
i\right) }\right \Vert \left( x,s\right) d\mu ds<\infty .%
\end{array}
\label{4001b}
\end{equation}

By Lemma \ref{lem3-3}, and direct calculation shows that%
\begin{equation*}
\begin{array}{lll}
\left( \frac{\partial }{\partial t}-\Delta _{b}\right) \left \Vert \sigma
^{\left( i\right) }\right \Vert \left( x,t\right) & \leq & \left \Vert
\sigma _{0}^{\left( i\right) }\right \Vert .%
\end{array}%
\end{equation*}

Since$\left \Vert \sigma _{0}^{\left( i\right) }\right \Vert \leq
\left
\Vert \nabla \eta _{0}^{\left( i\right) }\right \Vert $ and $%
\left
\Vert \eta _{0}^{\left( i\right) }\right \Vert \left( x,t\right) $
satisfy \ref{4016}, by Schauder estimates \cite{si} we have for any $\tilde{%
\varepsilon}>0$, there exists $n_{\tilde{\varepsilon}}>0$ such that$%
\left
\Vert \sigma _{0}^{\left( i\right) }\right \Vert \leq \left \Vert
\nabla \eta _{0}^{\left( i\right) }\right \Vert \leq $ $\tilde{\varepsilon}$
for any $i$ $\geq n_{\tilde{\varepsilon}}$. This shows that for any $\left(
x,t\right) \in \lbrack 0,T)$%
\begin{equation*}
\begin{array}{lll}
\left( \frac{\partial }{\partial t}-\Delta _{b}\right) \left \Vert \sigma
^{\left( i\right) }\right \Vert \left( x,t\right) +e^{T-t}\tilde{\varepsilon}
& \leq & 0.%
\end{array}%
\end{equation*}

We define $v^{\left( i\right) }\left( x,t\right) $ as follow%
\begin{equation*}
\begin{array}{l}
v^{\left( i\right) }\left( x,t\right) =\int_{M}p\left( x,y,t\right) \left
\Vert \Lambda \bar{\partial}_{b}\left( \rho ^{\left( i\right) }Ric\right)
\right \Vert \left( y\right) d\mu \left( y\right) .%
\end{array}%
\end{equation*}

Duo to (\ref{4001b}) and maximum principle we have%
\begin{equation*}
\begin{array}{l}
\left \Vert \sigma ^{\left( i\right) }\right \Vert \left( x,t\right) +e^{T-t}%
\tilde{\varepsilon}\leq v^{\left( i\right) }\left( x,t\right) +e^{T-t}\tilde{%
\varepsilon}.%
\end{array}%
\end{equation*}%
Since torsion is vanishing, it implies $\bar{\partial}_{b}Ric=0$ and by
nonnegativity of Ricci curvature it follows that 
\begin{equation}
\begin{array}{lll}
\left \Vert \Lambda \bar{\partial}_{b}\left( \rho ^{\left( i\right)
}Ric\right) \right \Vert \left( y\right) & \leq & \frac{C}{R_{i}}\chi
_{B_{2R_{i}\backslash R_{i}}}S\left( y\right) .%
\end{array}
\label{4015}
\end{equation}%
Similarly as (\ref{4013}), (\ref{4015}) gives that $v^{\left( i\right)
}\rightarrow 0$ uniformly on any compact subset as $i\rightarrow \infty $.
Since $\tilde{\varepsilon}$ is arbitrary, we have 
\begin{equation*}
\left \Vert \sigma \right \Vert \left( x,t\right) =0.
\end{equation*}
Finally, as a result we have $\left( tu\right) _{t}\geq 0$ and then $%
u=o\left( t^{-1}\right) .$ This completes the proof.
\end{proof}

\appendix

\section{ {}}

In this appendix, we construct "nice" domains to avoid the possibility of
the bad regularity for heat solutions in the case of degenerated parabolic
systems. \ In fact, we will give a proof on existence and regularity result
for $\left( 1,1\right) $-form $\phi $ of the Lichnerowicz-subLaplacian heat
equation. In the proof of main theorem, one required some regularity of the
heat solution in order to prove the mixed terms $\left( \bar{\partial}%
_{b}^{\ast }\Lambda \bar{\partial}_{b}+conj\right) \phi $ of (\ref{4006})
vanishing ( then the monotonicity follows). While we construct heat solution
on complete manifolds with exhaustion domains, we need the interior
regularity at least $C^{2,\alpha }\left( \Omega _{\mu }\right) $ and
boundary regularity as continuous function in $C\left( \bar{\Omega}_{\mu
}\right) $. This requirement are needed for Arzela Ascoli theorem and
integration by part in (\ref{4012}). In semigroup method, better regularity
of evolution equation comes from the regularity of infinitesimal generator.

We denote $C^{2,\alpha }\left( \Omega ,\Lambda ^{1,1}\right) $ as $%
C^{2,\alpha }$ sections of $\Lambda ^{1,1}$ on bounded domain $\Omega $. In
our case, it is $\Delta _{H}$ on Banach space $C^{2,\alpha }\left( \Omega
,\Lambda ^{1,1}\right) \cap C\left( \bar{\Omega},\Lambda ^{1,1}\right) $.
Note $\Delta _{H}=-\frac{1}{2}\left( \square _{b}+\bar{\square}_{b}\right) $%
. Here we denote $u$ as solution of following Dirichlet problem%
\begin{equation}
\square _{b}\phi =g  \label{601}
\end{equation}%
for $g\in C^{\infty }\left( \Omega ,\Lambda ^{1,1}\right) $. First we state
some results :

\begin{thm}[Kohn]
Let $M$ be a strictly pseudoconvex CR $(2n+1)$-manifold. If $1\leq q\leq n-1$%
, then $\left \Vert \phi \right \Vert _{\frac{1}{2}}^{2}\leq C\left[ \left(
\square _{b}\phi ,\phi \right) +\left \Vert \phi \right \Vert _{0}^{2}\right]
$ for $\phi \in C^{\infty }\left( \Lambda ^{0,q}\right) $. $\left \Vert
.\right \Vert _{s}$ stands for the $L^{2}$ Sobolev norm of order $s$.
\end{thm}

\begin{rem}
1. From the hypothesis in above theorem it requires $n\geq 2$. When $n=1$,
one refers to \cite{J1}.

2. Even though the operator $\square _{b}$ is not $\Delta _{H}$, in \cite{J2}
(see p.146) they actually prove the case for $\alpha =0$. Moreover, we have $%
\Delta _{H}=\mathcal{L}_{\alpha }$ with $\alpha =0$ up to lower order terms.
Here $\mathcal{L}_{\alpha }=-\Delta _{b}+i\alpha T.$
\end{rem}

\begin{def}
\textrm{We say that a point $x$ in the boundary $\partial \Omega $
of a domain is a characteristic point if $\xi $ is tangent to
$\partial \Omega $ at $x.$}
\end{def}

The following is the interior and boundary regularity result by Jerison \cite%
{J1}.

\begin{thm}
\label{apx2}Let $U$ be the open subset of $M$ containing no characteristic
points of $\partial \Omega $. If $\psi ,\varphi \in C_{0}^{\infty }\left(
U\right) $, $\psi =1$ in the neighborhood of the support of $\varphi $, and $%
u$ satisfies (\ref{601}) with $\psi g\in \Gamma _{\beta }\left( \bar{\Omega}%
,\Lambda ^{0,q}\right) $, then $\varphi \phi \in \Gamma _{\beta +2}\left( 
\bar{\Omega},\Lambda ^{0,q}\right) $ and 
\begin{equation*}
\left \Vert \varphi \phi \right \Vert _{\Gamma _{\beta +2}}\leq c\left(
\left \Vert \psi g\right \Vert _{\Gamma _{\beta }}+\left \Vert \psi \phi
\right \Vert _{L^{2}}\right) .
\end{equation*}
\end{thm}

When an isolated characteristic boundary point occurs, Jerison proved the
regularity result when the neighborhood have strictly convexity property.
The convexity is defined by Folland-Stein local coordinates $\Theta \left(
p,-\right) :U\rightarrow \mathrm{H}^{n}$, and the boundary near point $p$ is
corresponding to graph $\tilde{t}=\sum \alpha _{i}\tilde{x}_{i}^{2}+\beta
_{j}\tilde{y}_{j}^{2}+e\left( \tilde{x},\tilde{y}\right) $, where $e\left( 
\tilde{x},\tilde{y}\right) =O\left( \left \vert \tilde{x}\right \vert
^{3}+\left \vert \tilde{y}\right \vert ^{3}\right) $. Strictly convex means $%
\alpha _{i},\beta _{j}>0$ (see eq. (7.4) and A.3 in \cite{J2}). In the
following we state the theorem in the form we want. Reader who is confused
can refer to theorem 7.6, Proposition 7.11, and Corollary 10.2 in \cite{J2}.

\begin{thm}
\label{Ap1} Let $p$ be an isolated characteristic point on $\partial \Omega $
and in some neighborhood $U_{p}$ of $p$ the geometry $U_{p}\cap \Omega $ is
like the domain $\left \{ \left( x,y,t\right) :M_{c}\left( \left \vert
x\right \vert ^{2}+\left \vert y\right \vert ^{2}\right) <t\right \} $ in
the Heisenberg group, where $M_{c}$ a positive number . Then $\varphi \phi
\in \Gamma _{\beta +2}\left( \bar{\Omega},\Lambda ^{0,q}\right) $, where the
best $\beta $ depends on $M_{c}$. Moreover, as $M_{c}\nearrow \infty $, one
can choose $\beta \nearrow \infty $.
\end{thm}

\begin{rem}
In Theorem \ref{Ap1}, one required $g\in \Gamma _{\beta }\left( \bar{\Omega}%
,\Lambda ^{0,q}\right) $ for $\beta >2$. Moreover, $\beta $ has upper bound $%
\beta _{0}-2$ , where $\beta _{0}$ is an index related to the geometry of
the boundary. In \cite{J2}, they proved $M_{c}\nearrow \infty $, then $\beta
_{0}\nearrow \infty $.
\end{rem}

In order to construct a $C^{2,\alpha }$ Lichnerowitz-subLaplacian heat
solution, we need the exhaustion domain which satisfy the property above. In
the following we prove that it is possible by perturbing the boundary of
exhaustion domain.

\begin{thm}
\label{nice} For any given positive number $M_{c}$, there exists exhaustion
domains $\Omega _{\mu }$ such that $\partial \Omega _{\mu }$ consist only
isolated characteristic points with property as in Theorem \ref{Ap1} with
given $M_{c}$.
\end{thm}

\begin{proof}
We construct the exhaustion domain with smooth boundary arbitrarily. Since $%
\partial \Omega _{\mu }$ is compact, we define $\Xi _{\mu }$ the set
consisting all the characteristic points. Then the closure of $\Xi _{\mu }$
is compact. At each point there exist coordinate $V_{p}$ such that we can
express the boundary as $r\left( z,t\right) =t-q\left( z\right) +e\left(
x,y\right) $ in $B_{p}\left( \varepsilon _{p}\right) $ for some $\varepsilon
_{p}$ depend on $p$, where $q\left( z\right) =\alpha _{i}x_{i}^{2}+\beta
_{j}y_{j}^{2}$ for some real numbers $\alpha _{i},\beta _{j}$. Since
injective radius (with respect to \ some adapted metric) is uniformly
bounded below on $\partial \Omega _{\mu }$, $\varepsilon _{p}$ can be chosen
to not depend on $p$ but $\mu $ only. These Folland-Stein coordinate
neighborhoods form an open covering for $\bar{\Xi}_{\mu }$.

Now we claim there is a small modification to boundary so that $\bar{\Xi}%
_{\mu }$ contains only isolated characteristic points.

Assume $B_{p_{i}}\left( \varepsilon \right) $ are the covering of $\bar{\Xi}%
_{\mu }$, we can choose $\varepsilon _{1}<\varepsilon _{2}<\varepsilon $
such that $B_{p_{i}}\left( \varepsilon _{1}\right) $ are still a covering of 
$\bar{\Xi}_{\mu }$. We start at point $p_{1}$. First we deform the graph in
the coordinate of $B_{p_{1}}\left( \varepsilon _{1}\right) $ to plane $t=0$
and smoothly attached to graph on $\partial B_{p_{i}}\left( \varepsilon
_{2}\right) $. Under the deformation we keep point $p_{1}$ as the only
characteristic point. This is possible by noticing that we only need to take 
$q\left( z\right) $ into consideration ( because this term dominate all the
other inside small ball. ) and we only need to consider the case in the
Heisenberg group with graph $t=q\left( z\right) $ in $B_{p_{i}}\left(
\varepsilon \right) $. We modify $q\left( z\right) $ into new one $\tilde{q}%
\left( z\right) $ by define $\tilde{q}\left( z\right) =-\max
\limits_{\left
\vert z\right \vert =\varepsilon _{2}}q\left( z\right) $ in $%
B_{p_{i}}\left( \varepsilon _{1}\right) $ and $\varphi \left( \left \vert
z\right \vert ,\theta \right) $ in $B_{p_{i}}\left( \varepsilon _{2}\right)
\backslash B_{p_{i}}\left( \varepsilon _{1}\right) $ where $\varphi \left(
\left \vert z\right \vert ,\theta \right) $ is a smooth monotone function in 
$\left
\vert z\right \vert $ for each $\theta $ such that the function
smoothly attached to the value $q\left( z\right) $ on $\partial
B_{p_{i}}\left( \varepsilon _{2}\right) $ and $\tilde{q}\left( z\right)
=q\left( z\right) $ on $B_{p_{i}}\left( \varepsilon \right) \cap
B_{p_{i}}^{c}\left( \varepsilon _{2}\right) $. This modification clearly
imply the origin is the only characteristic point in $B_{p_{i}}\left(
\varepsilon _{1}\right) $. Moreover, we can choose $\varphi \left(
\left
\vert z\right \vert ,\theta \right) $ very steep so that all the
point $(z,q\left( z\right) )$ for $z\in B_{p_{1}}\left( \varepsilon
_{2}\right) \backslash \left( 0,0\right) $ are noncharacteristic. We define
the new domain as $\Omega _{\mu ,1}$. Specifically, 
\begin{equation*}
\Omega _{\mu ,1}=\left \{ \Omega _{\mu }\backslash B_{p_{1}}\left(
\varepsilon _{2}\right) \right \} \cup \left( M\cap \left \{ \left(
z,t\right) :t>\tilde{q}\left( z\right) -R\left( z,t\right) \text{ for }z\in
B_{p_{1}}\left( \varepsilon _{2}\right) \right \} \right) .
\end{equation*}

Then we continue the same process on $p_{2}$, and the new domain is $\Omega
_{\mu ,2}$. Observe that the process do not create new characteristic points
but eliminate all the characteristic point inside $B_{p_{i}}\left(
\varepsilon _{2}\right) $ except $p_{i}$. Continuing this process we are
able to deform domain $\Omega _{\mu }$ into new one that only consist
isolated characteristic points on the boundary with $M_{c}=0.$

To modify $M_{c}$ into any value we want is easier. One can do the same
process by deforming the graph into parabolic.
\end{proof}

For convenience, we call the domain in above theorem as sweetsop domain.

\begin{rem}
\bigskip The above theorem can be simplified if we can construct strictly
convex domain in $M$. But the existence to this kind of exhaustion domain
isn't known yet.
\end{rem}

We recall theorems from semigroup method. For the definition of analytic
semigroup, one can refer to definition 12.30 in \cite{r} (\cite{p}). We
cited the characterization of infinitesimal generator of analytic
semigroups. Notation here $X$ is Banach space and $A$ is operator defined on 
$X.$ Note $A$ can be unbounded operator $A:D\left( A\right) \rightarrow X$,
where $D\left( A\right) $ is a subset in $X$ such that $Ax$ can be defined.
As before, we denote $\Gamma _{\beta }$ the Lipschitz classes associated to
nonisotropic distance (refered \cite{J2}) and $\Gamma _{\beta }\left( \bar{%
\Omega},\Lambda ^{1,1}\right) $ the restriction to $\bar{\Omega}$ of
sections of $\Lambda ^{1,1}$ with coefficients in $\Gamma _{\beta }\left( 
\bar{\Omega}\right) $. We denote $\left \Vert {}\right \Vert _{\Gamma
_{\beta }}$ the norm of Banach space $\Gamma _{\beta }\left( \bar{\Omega}%
,\Lambda ^{1,1}\right) $, and $R_{\lambda }\left( A\right) $ as the inverse
operator of $A_{\lambda }:=A-\lambda I$ as $A_{\lambda }$ is one-to-one. The
resolvent set of the operator $A$ is the subset of $%
\mathbb{C}
$ that $R_{\lambda }\left( A\right) $ exists, bounded, and the domain is
dense in $X$. When we apply, we let $X=\Gamma _{\beta }\left( \bar{\Omega}%
,\Lambda ^{1,1}\right) $ and $A=\Delta _{H}$. Here we state general theorems
for following evolution systems 
\begin{equation*}
\dot{u}=Au+f
\end{equation*}%
where $f\in $ $X$.

\begin{thm}
(\cite[Theorem 12.31]{r}) A closed, densely defined operator $A$ in $X$ is
the generator of an analytic semigroup if and only if there exists $\omega $
a real number such that the half-plane $\text{Re}\lambda >\omega $ is
contained in the resolvent set of $A$ and, moreover, there is a constant $C$
such that 
\begin{equation}
\left \Vert R_{\lambda }\left( A\right) \right \Vert \leq \frac{C}{\left
\vert \lambda -\omega \right \vert }  \label{602}
\end{equation}%
for $\text{Re}\lambda >\omega $ and $\left \Vert .\right \Vert $ is the norm
of $X$.
\end{thm}

\begin{thm}
(\cite[Theorem 12.33]{r}) \label{apx1}Let $A$ be the infinitesimal generator
of an analytic semigroup and assume that the spectrum of $A$ is entirely to
the left of the line $\text{Re}\lambda =\omega $. Then there exists a
constant $M$ such that 
\begin{equation*}
\left \Vert e^{At}\right \Vert \leq Me^{\omega t},
\end{equation*}%
where $\left \Vert .\right \Vert $ is the norm of $X$.
\end{thm}

One can refer to section 7.1 in \cite{p} or page 421 in \cite{e} for the
application of semigroup theory for $A$ is a strong ellieptic operator. In
our case, the missing boundary regularity is replaced by theorem \ref{Ap1}
(refer to \cite{gv}, \cite{gv2}). Then one can follow Stewart \cite{sh} and
consider H\"{o}lder spaces as interpolation space \cite{la} to obtain the
resolvent estimates \ref{602}. As a result, the regularity of the parabolic
systems follows by theorem \ref{apx1}.

In conclusion, we are able to choose exhaustion domain with $\beta _{0}$
large enough, then follow theorem above, we can choose $\beta $ large enough
to make sure the function space $X$ is contained in $C^{2,\alpha }$. This is
possible by relation $C^{\beta }\subset \Gamma _{\beta }\subset C^{\beta /2}$
as in 20.5, 20.6 of \cite{fs}. This completes the proof of Proposition \ref%
{prop2-1}.

\end{document}